\documentclass[10pt]{article}

 \usepackage{amsmath}
\usepackage{amsthm,amssymb}
 \usepackage{makeidx,epsfig,lscape}
 \usepackage{color,colortbl}
 \usepackage{fancyhdr}
 \usepackage{times}
 \usepackage{xcolor,pict2e}
 \usepackage[latin1]{inputenc}
 
\newcommand{\R}{\mathbb R}
\newcommand{\N}{\mathbb N}

\newcommand{\Qv}{\mathbb Q}

\newcommand{\E}{\mathbb E}

 \renewcommand{\headrulewidth}{0pt}
 \renewcommand{\footrulewidth}{0.5pt}

 \definecolor{myaqua}{rgb}{0.0,0.5,0.55}
 \definecolor{lightaqua}{rgb}{0.75,0.95,0.95}

 \usepackage[colorlinks = true,
            linkcolor = myaqua,
            urlcolor  = blue,
            citecolor = red,
            pdfpagemode=UseOutlines,
            bookmarksnumbered=true,
            pdfpagelabels,
            breaklinks]{hyperref}

\usepackage{caption}
\usepackage{floatrow}

\newtheorem{theorem}{Theorem}
\newtheorem{prop}{Proposition}
\newtheorem{lem}{Lemma}
\newtheorem{coro}{Corollary}
\newtheorem{defn}{Definition}[section]
\newtheorem{exple}{Example}[section]

\newtheorem{hyp}{Hypothesis}[section]
\newtheorem{rem}{Remark}[section]
\def\lin#1#2{\textcolor[rgb]{0.6,0.6,0.6}{\vspace*{#1mm} \hrule
   height 3 pt \vspace*{#2mm}}}
\def\bt{\begin{tabular}}
\def\et{\end{tabular}}
\def\and{\mbox{ and }}
\def\E{\mbox{\bf E}}
\def\P{\mbox{\bf P}}

\def\1{{\bf 1}}

 \def\boxx#1#2#3#4#5{
 {\linethickness{#4pt}\put(#1,#5){\color{myaqua}{\line(1,0){#3}}}}
 \multiput(#1,#2)(0,#4){2}{\line(1,0){#3}}
 \multiput(#1,#2)(#3,0){2}{\line(0,1){#4}}
  }

\begin{document}

 $\mbox{ }$

 \vskip 12mm

{ 

{\noindent{\Large\bf\color{myaqua}
  Stochastic differential equations driven by relative martingales} 
%
\\[6mm]
{\bf  Fulgence EYI OBIANG$^{1,a}$, Paule Joyce MBENANGOYE$^{1,b}$, Ibrahima FAYE$^{2,c}$ and Octave MOUTSINGA$^{1,d}$}}
\\[2mm]
{ 
$^1$URMI Laboratory, Département de Mathématiques et Informatique, Faculté des Sciences, Université des Sciences et Techniques de Masuku, BP: 943 Franceville, Gabon. 
\\
$^2$Université Alioune Diop de Bambey, Sénégal.\\
$^a$ Email: \href{mailto:feyiobiang@yahoo.fr}{\color{blue}{\underline{\smash{feyiobiang@yahoo.fr}}}}\\[1mm]
$^b$ Email: \href{mailto:paulejoycembenangoye@yahoo.fr}{\color{blue}{\underline{\smash{paulejoycembenangoye@yahoo.fr}}}}\\[1mm]
$^c$ Email: \href{mailto:Ibou.faye@uadb.edu.sn}{\color{blue}{\underline{\smash{Ibou.faye@uadb.edu.sn}}}}\\[1mm]
$^d$ Email: \href{mailto:octavemoutsing-pro@yahoo.fr}{\color{blue}{\underline{\smash{octavemoutsing-pro@yahoo.fr}}}}\\[1mm]
\lin{5}{7}

 {  
 {\noindent{\large\bf\color{myaqua} Abstract}{\bf \\[3mm]
 \textup{
 This paper contributes to the study of relative martingales. Specifically, for a closed random set $H$, they are processes null on $H$ which decompose as $M=m+v$, where $m$ is a càdlàg uniformly integrable martingale and, $v$ is a continuous process with integrable variations such that $v_{0}=0$ and $dv$ is carried by $H$. First, we extend this notion  to stochastic processes not necessarily null on $H$, where $m$ is considered local martingale instead of a uniformly integrable martingale. Thus, we provide a general framework for the new larger class of relative martingales by presenting some structural properties. Second, as applications, we construct solutions for skew Brownian motion equations using continuous stochastic processes of the above mentioned new class. In addition, we investigate stochastic differential equations driven by a relative martingale.
 }}
 \\[4mm]
 {\noindent{\large\bf\color{myaqua} Keywords:}{\bf \\[3mm]
 Relative martingales; Skew Brownian motion; class $(\Sigma)$; Stochastic differential equations; Signed measure theory.
}}\\[4mm]{\noindent{\large\bf\color{myaqua} MSC:}{\color{blue} 60G07; 60G20; 60G46; 60G48}}
\lin{3}{1}

\renewcommand{\headrulewidth}{0.5pt}
\renewcommand{\footrulewidth}{0pt}

 \pagestyle{fancy}
 \fancyfoot{}
 \fancyhead{} 
 \fancyhf{}
 \fancyhead[RO]{\leavevmode \put(-160,0){\color{myaqua} Fulgence EYI OBIANG et al. (2022)} \boxx{15}{-10}{10}{50}{15} }
 \fancyfoot[C]{\leavevmode
 \put(-2.5,-3){\color{myaqua}\thepage}}

 \renewcommand{\headrule}{\hbox to\headwidth{\color{myaqua}\leaders\hrule height \headrulewidth\hfill}}
\section*{Introduction}
In the theory of zeros of continuous martingales \cite{1}, Azéma and Yor have introduced two remarkable classes of processes respectively named: $\mathcal{R}(\mathcal{H})$ and  $\mathcal{R}$. More precisely, they are processes satisfying the next both definitions:
\begin{defn}[\textbf{Definition 2.1 of Azéma and Yor\cite{1}}]\label{d01}
Let  $\mathcal{H}$ be a random  optional closed set. We call $\mathcal{R}(\mathcal{H})$ the class of processes $(X_{t};t\geq0)$ vanishing on $\mathcal{H}$ and admitting a decomposition of the form
$$X_{t}=M_{t}+A_{t},$$
where $(M_{t};t\geq0)$ is a right continuous uniformly integrable martingale, $(A_{t};t\geq0)$ is a continuous and adapted variation integrable process such that $dA$ is carried by $\mathcal{H}$.
\end{defn} 
 
\begin{defn}[\textbf{Definition 2.2 of Azéma and Yor\cite{1}}]\label{d02}
 We call $\mathcal{R}$ the class of processes $(X_{t};t\geq0)$ admitting a decomposition of the form
$$X_{t}=M_{t}+A_{t},$$
where $(M_{t};t\geq0)$ is a right continuous uniformly integrable martingale, $(A_{t};t\geq0)$ is a continuous and adapted variation integrable process such that $dA$ is carried by $H=\{t\geq0:X_{t}=0\}$.
\end{defn} 
Meyer called processes of the  class $\mathcal{R}(\mathcal{H})$, relative martingales because they are true martingales outside of the random set $\mathcal{H}$. Remark from Definition \ref{d01} that all relative martingales vanish on $\mathcal{H}$. This allows to see that $\mathcal{R}(\mathcal{H})\subset\mathcal{R}$. These both classes have been extensively studied in \cite{1}.

On another hand, Yor has extended the notion of class $\mathcal{R}$ to semi-martingales by introducing in \cite{y1}, an another remarkable larger class $(\Sigma)$ of processes. Specifically, they are processes $X$ which decompose as $X=M+A$, where $M$ is a càdlàg local martingale with $M_{0}=0$ and $A$ is a finite variation process such that the signed measure $dA$ is carried by $\{t\geq0:X_{t}=0\}$. Such stochastic processes are strongly related to many studies in probability theory. For instance, they play a capital role in the theory of Azéma-Yor martingales, the study of zeros of continuous martingales \cite{1}, the study of Brownian local times, the balayage formulas for the progressive case \cite{mey}. They are used to resolve Skorokhod's reflection equation and embedding problem . This class has been studied extensively in several studies, enriching the general framework by deriving characterization results, by studying their main properties, presenting their applications, and relaxing more and more the original hypotheses. Note that the class $(\Sigma)$ contains the other two above mentioned classes. However, remark that $\mathcal{R}(\mathcal{H})$ is include in the class $\mathcal{R}$ and the class $(\Sigma)$ only because the fact that all elements of the class $\mathcal{R}(\mathcal{H})$ vanish on $\mathcal{H}$. Thus, if we remove this cancellation condition on $\mathcal{H}$ we lose the inclusion of $\mathcal{R}(\mathcal{H})$ in $(\Sigma)$. 

The aim of this paper is to extend the notion of class $\mathcal{R}(\mathcal{H})$ to càdlàg processes not necessarily null on $\mathcal{H}$ and whose the martingale part is not necessarily uniformly integrable. We do this by considering a new class that we term $\mathcal{M}(\mathcal{H})$ and define as follows:
\begin{defn}\label{d1}
 We shall say that a process $M$ is a relative martingale $(M\in\mathcal{M}(H))$ if it decomposes as $M=m+v$, where
\begin{enumerate}
	\item $m$ is a càdlàg  local martingale, with $m_{0}=0$;
	\item $v$ is an adapted continuous process with finite variations such that $v_{0}=0$; 
	\item  $\int{1_{\mathcal{H}^{c}}(s)dv_{s}}=0$.
\end{enumerate}
\end{defn}

Admittedly, this new class is not a subset of the class $(\Sigma)$ and reciprocally. But, it also contains interesting examples playing a key role in the stochastic analysis. For instance, if $\mathcal{H}$ is the set of zeros of a standard Brownian motion $D$, hence for an other Brownian motion $B$ independent of $D$,  the geometric Itô-Mckean skew Brownian motion with Azzalini skew normal distribution 
$$X^{\delta}=\sqrt{1-\delta^{2}}B+\delta|D|$$
is a process of the class $\mathcal{M}(\mathcal{H})$. This process is used by several authors. For instance, Corns and Satchell () and Zhu and He [24] worked on this type of skew Brownian motion and priced European style options. Recently, in the preprint (), the authors consider an asset evolving as $X^{\delta}$ to formulate the wealth function under continuous time investment strategy of insurance companies. In this last mentioned reference, the authors investigate the next stochastic differential equation:
$$dS_{t}=\left(\mu+\frac{\sigma^{2}}{2}\right)S_{t}dt+\sigma S_{t}dX^{\delta}_{t}.$$
Thus, it would be useful to provide a general framework and develop techniques to manipulate the processes of this new class of relative martingales. This could open new perspectives in applications and in other areas of probability theory.

The  remainder of this paper is organized as follows. In Section \ref{sec:1}, we present some useful preliminaries. Section \ref{sec:2} is devoted to the study of the class $\mathcal{M}(\mathcal{H})$, where we give some examples and explore some general properties. Section \ref{sec:3} focuses on the construction of solutions for skew Brownian motion equations using stochastic processes of the class $\mathcal{M}(\mathcal{H})$.  Finally, in Section \ref{sec:4}, we investigate stochastic differential equations driven by a relative martingale.

\section{Notations and recall of some useful results}\label{sec:1}

In this section, we provide notations and recall some useful results that will be used throughout this work. Thus, we start by giving some notations.

\subsection{Notations}

Throughout we fix a filtered probability space $\left(\Omega,(\mathcal{F}_{t})_{t\geq0},\mathcal{F},\P\right)$ satisfying the usual conditions. We shall always consider that $H$ is the zero set of a continuous martingale $D$. And we shall use the following notations:
\begin{itemize}
	\item $\forall t\geq0$, $\gamma_{t}=\sup\{s\leq t:D_{s}=0\}$;
	\item $\gamma=\sup\{t\geq0:D_{s}=0\}$;
	\item For any other process $X$, we will denote $g_{t}=\sup\{s\leq t:X_{s}=0\}$ and $g=\sup\{t\geq0:X_{s}=0\}$;
	\item $\Qv=\frac{|D_{\infty}|}{\E(|D_{\infty}|)}\P$.
\end{itemize}
We consider in this paper that 
$$\P(0<\gamma<\infty)=1.$$
Remark that the random time $\gamma$ is not a stopping time with respect to $(\mathcal{F}_{t})_{t\geq0}$ but an honest time. Hence, we shall denote $(\mathcal{G}_{t})_{t\geq0}$ the smallest right continuous filtration containing $(\mathcal{F}_{t})_{t\geq0}$ for which $\gamma$ is a stopping time.

On another hand, it is known that for any continuous semi-martingale $Y$, the set $\mathcal{W}=\{t\geq0; Y_{t}=0\}$ cannot be ordered. However, the set $\R_{+}\setminus\mathcal{W}$ can be decomposed as a countable union $\cup_{n\in\N}{J_{n}}$ of intervals $J_{n}$. Each interval $J_{n}$ corresponds to some excursion of $Y$. In other words, if $J_{n}=]g_{n},d_{n}[$, $Y_{t}\neq0$ for all $t\in]g_{n},d_{n}[$ and $Y_{g_{n}}=Y_{d_{n}}=0$. For any constant $\alpha\in[0,1]$, we consider a sequence $(\xi_n)$ of i.i.d. Bernoulli variables such that
$$\P(\zeta_{n}=1)=\alpha\text{ and }P(\zeta_{n}=-1)=1-\alpha.$$
Now, let us define the process $Z^{Y}$ as follows.
\begin{equation}\label{zalpha}
	Z^{Y}_{t}=\sum_{n=0}^{+\infty}{\zeta_{n}1_{]g_{n},d_{n}[}(t)}.
\end{equation}

If we assume that $\alpha$ is a piecewise constant function associated with a partition $(0=t_{0}<t_{1}<\cdots<t_{n-1}<t_{m})$, i.e., $\alpha$ is of the form
$$\alpha(t)=\sum_{i=0}^{m}{\alpha_{i}1_{[t_{i},t_{i+1})}(t)},$$
where $\alpha_{i}\in[0,1]$ for all $i=0,1,\cdots,m$, then we shall consider the process
\begin{equation}\label{Zalpha}
	\mathcal{Z}^{Y}_{t}=\sum_{n=0}^{+\infty}{\sum_{i=0}^{m}{\zeta^{i}_{n}1_{]g_{n},d_{n}[\cap[t-{i},t_{i+1})}(t)}},
\end{equation}
where $(\zeta^{i}_{n})_{n\geq0}$, $i=1,2,\cdots,m$, are $m$ independent sequences of independent variables such that
$$\P(\zeta^{i}_{n}=1)=\alpha_{i}\text{ and }\P(\zeta^{i}_{n}=-1)=1-\alpha_{i}.$$
\subsection{Some useful results of enlargement filtrations}

Now, we shall recall some results of the theory of enlargement filtrations which are  useful in the current work.
\begin{prop}[\textbf{Proposition 3.1 of Azéma and Yor\cite{1}}]\label{rho}

Let $H$ be a random  optional closed set. Denote $g=\sup{H}$ and represent by $(\mathcal{G}_{t})_{t\geq0}$, the progressive enlargement of the filtration $(\mathcal{F}_{t})_{t\geq0}$ with respect to $g$. Let $(V_{t})_{t\geq0}$ be a $(\mathcal{G}_{g+t})_{t\geq0}-$   optional process. There exists a unique $(\mathcal{F}_{t})_{t\geq0}-$   optional process $(U_{t})_{t\geq0}$  which vanishes on $H$ such that $\forall t\geq0$, $U_{g+t}=V_{t}$ and $U_{0}=V_{0}$ on $\{g=0\}$. That defines a function
$\rho:V\longmapsto U$. $\rho$ is linear, non-negative and preserves products.
\end{prop}       

\begin{theorem}\label{qot}[\textbf{Quotient theorem: Theorem 3.2 of Azéma and Yor \cite{1}}]
\begin{enumerate}
	\item If $(X_{t};t\geq0)$ is a stochastic process of the class $\mathcal{R}(H)$, hence, the process $(\chi_{t};t>0)$ defined by $\chi_{t}=\frac{X_{g+t}}{Y_{g+t}}$
is a $\left(Q, (\mathcal{G}_{g+t}){t>0}\right)$ uniformly integrable martingale.
\item Reciprocally, let $(\chi_{t};t>0)$ a $\left(Q, (\mathcal{G}_{g+t}){t>0}\right)$ uniformly integrable martingale; the stochastic process $X=(Y_{t}\rho(\chi_{\cdot})_{t};t\geq0)$ is the unique process of $\mathcal{R}(H)$ such that $\chi_{t}=\frac{X_{g+t}}{Y_{g+t}}$ for all $t>0$.
\end{enumerate}
\end{theorem}

\begin{theorem}\label{gir}[\textbf{Theorem 4.1.2 of Azéma and Yor \cite{1}}]
Let $X=m+v$ be a $(\P,(\mathcal{F}_{t})_{t\geq0})$- semi-martingale, where $m$ is a local martingale and $v$ is a process with finite variations such that $dv$ is carried by $H$. The process $\widetilde{X}$ defined by
$$\widetilde{X}_{t}=X_{\gamma+t}-X_{\gamma}-\int_{\gamma}^{\gamma+t}{\frac{d\langle X,|D|\rangle_{s}}{|D_{s}|}}$$
is then a $(\Qv,(\mathcal{G}_{\gamma+t})_{t\geq0})$- local martingale and $\langle \widetilde{X},\widetilde{X}\rangle_{t}=\langle X,X\rangle_{\gamma+t}-\langle X,X\rangle_{\gamma}$.
\end{theorem}

Remark that $\widetilde{X}$ holds the next:
$$\widetilde{X}_{t}=X_{\gamma+t}-X_{\gamma}-\int_{\gamma}^{\gamma+t}{\frac{d\langle X,D\rangle_{s}}{D_{s}}}.$$

\begin{lem}[\textbf{Lemma 5.7 of Jeulin \cite{jeu}}]\label{j80}
Let $g$ be an honest variable with respect to $(\mathcal{F}_{t})_{t\geq0}$. Let $(\mathcal{G}_{t})_{t\geq0}$ be the progressive enlargement of the filtration $(\mathcal{F}_{t})_{t\geq0}$ with respect to $g$. If $\tau$ is a stopping time with respect to $(\mathcal{G}_{t})_{t\geq0}$ such that $g<\tau$ on $\{g<\infty\}$, hence
$$\mathcal{G}_{\tau}=\mathcal{F}_{\tau}.$$
\end{lem}

\subsection{Recall of some useful balayage formulas}

Balayage formulas are power tools in this work. In next, we recall some balayage results we use in this paper. Thus, we start by the predictable case for continuous semimartingales.
\begin{prop}
Let $X$ be a continuous semimartingale and define $g_{t}=\sup\{s\leq t:X_{s}=0\}$. If $k$ is a locally bounded predictable process, then
$$k_{g_{t}}X_{t}=k_{g_{0}}X_{0}+\int_{0}{k_{g_{s}}dX_{s}}.$$
\end{prop}

In next, we provide the result for càdlàg semimartingales.

\begin{prop}
Let $X$ be a continuous semimartingale and define $g_{t}=\sup\{s\leq t:X_{s}=0\}$. If $k$ is a bounded predictable process, then
$$k_{g_{t}}X_{t}=k_{g_{0}}X_{0}+\int_{0}{k_{g_{s}}dX_{s}}.$$
\end{prop}

The balayage formulas for continuous semi-martingales in the progressive case, are critical tools in this study. We recall these results below.

\begin{prop}\label{balpro}
Let $Y$ be a continuous semi-martingale and $\gamma^{'}_{t}=\sup\{s\leq t:Y_{s}=0\}$. Let $k$ be a bounded progressive process, where ${^{p}k_{\cdot}}$ denotes its predictable projection. Then,
$$k_{\gamma^{'}_{t}}Y_{t}=k_{0}Y_{0}+{\int_{0}^{t}{^{p}k_{\gamma^{'}_{s}}dY_{s}}+R_{t}},$$
where $R$ is an adapted, continuous process with bounded variations such that $dR_{t}$ is carried by the set $\{Y_{s}=0\}$.
\end{prop}

Proposition \ref{balpro} is a powerful and interesting tool. However, the fact that we know nothing about the form of the process $R$ can be limiting. The processes $Z^{Y}$ and $\mathcal{Z}^{Y}$ are critical to this study. Bouhadou and Ouknine \cite{siam} identified the process $R$ of Proposition \ref{balpro} when the progressive process $k$ is equal to $Z^{\alpha}$ or $\mathcal{Z}^{\alpha}$. We recall these results below.  

\begin{prop}[\textbf{Ouknine and Bouhadou \cite{siam}}]\label{pzalph}
Let $Y$ be a continuous semi-martingale and $Z^{Y}$ be the process defined in \eqref{zalpha}. Then,
$$Z^{Y}_{t}Y_{t}=\int_{0}^{t}{Z^{Y}_{s}dY_{s}}+(2\alpha-1)L_{t}^{0}(Z^{Y}Y),$$
where $L_{\cdot}^{0}(Z^{Y}Y)$ is the local time of the semi-martingale $Z^{Y}Y$.
\end{prop}

\begin{prop}[\textbf{Ouknine and Bouhadou \cite{siam}}]\label{siam}
Let $Y$ be a continuous semi-martingale and $Z^{Y}$ be the process defined in \eqref{zalpha}. Then,
$$\mathcal{Z}^{Y}_{t}Y_{t}=\int_{0}^{t}{\mathcal{Z}^{Y}_{s}dY_{s}}+\int_{0}^{t}{(2\alpha(s)-1)dL_{s}^{0}(\mathcal{Z}^{Y}Y)},$$
where $L_{\cdot}^{0}(\mathcal{Z}^{Y}Y)$ is the local time of the semi-martingale $\mathcal{Z}^{Y}Y$.
\end{prop}

\section{A general framework for a larger family of relative martingales}\label{sec:2}

In this section, we bring contributions to the study of stochastic processes of the form: $M=m+v$, where $m$ is a martingale and $v$ is a process with finite variations such that $dv$ is carried by $H$. Note that a known subfamily of such processes is the class of relative martingales, $\mathcal{R}(H)$. Here, we extend this notion of relative martingales to semimartingales which don't necessary vanish on $H$ and whose the martingale part is not necessary uniformly integrable. More precisely, we  provide a general framework to a larger class of processes that we shall name, class $\mathcal{M}(H)$.

\subsection{Some examples}

Now, we shall provide some examples of the class $\mathcal{M}(H)$. First remark a natural example which is, the process $M=|D|$. In fact, all processes of  the class $\mathcal{R}(H)$ and all elements, $X$ of the class $(\Sigma)$ such that $\{t\geq0:X_{t}=0\}\subset H$, are in the class $\mathcal{M}(H)$. However, there also exist stochastic processes which don't necessary vanish on $H$. In next, we provide some such examples.

\begin{exple}\label{ex1}

	 Any semimartingale $M=m+v$ such that $M_{0}=0$ and $DM-\langle D,M\rangle$ is a local martingale, is an element of the class $\mathcal{M}(H)$. Indeed, We have from integration by parts that 
	$$D_{t}M_{t}=\int_{0}^{t}{M_{s}dD_{s}}+\int_{0}^{t}{D_{s}dm_{s}}+\int_{0}^{t}{D_{s}dv_{s}}+\langle D,M\rangle_{t}.$$
	Hence, it follows that $\int_{0}^{t}{D_{s}dv_{s}}=0$. That is, $dv$ is carried by $H$. This proves that $M\in\mathcal{M}(H)$.
\end{exple}

\begin{exple}\label{ex2}

Let $m$ be a càdlàg local martingale which vanishes at zero and $X$ be a continuous process of the class $(\Sigma)$ such that $\{t\geq0:X_{t}=0\}\subset H$. Hence, the following processes are in the class $\mathcal{M}(H)$:
\begin{itemize}	
	\item $X^{1}=\min{(m,m-X)}$;
	\item $X^{2}=\min{(m,m+X)}$;
	\item $X^{3}=\min{(m-X,m+X)}$;
	\item $X^{4}=\max{(m,m-X)}$;
	\item $X^{5}=\max{(m,m+X)}$,
	\item $X^{6}=\max{(m-X,m+X)}$.
\end{itemize}
\end{exple}

\subsection{Some structural properties}

Here, we shall explore some general properties satisfied by stochastic processes of the class  $\mathcal{M}(H)$. Hence, we start by the next remark: 

\begin{rem}\label{r1}
	The class $\mathcal{M}(H)$ is a vector space.
\end{rem}

In what follows, we derive some properties related to the notion of stochastic integral.
\begin{lem}\label{l1}
Let $M=m+v$ be a process of the class $\mathcal{M}(H)$. The following hold:
\begin{enumerate}
  \item For any locally bounded predictable process $h$, $\int_{0}^{\cdot}{h_{s}dM_{s}}$ is an element of the class $\mathcal{M}(H)$.
	\item If $h$ is a locally bounded predictable process null on $H$. Hence, $\int_{0}^{\cdot}{h_{s}dM_{s}}$  and $\int_{0}^{\cdot}{h_{s-}dM_{s}}$ are local martingales. 
\end{enumerate}
\end{lem} 
\begin{proof}
 We have $\forall t\geq0$,
	$$\int_{0}^{t}{h_{s}dM_{s}}=\int_{0}^{t}{h_{s}dm_{s}}+\int_{0}^{t}{h_{s}dv_{s}}$$
	and
	$$\int_{0}^{t}{h_{s-}dM_{s}}=\int_{0}^{t}{h_{s-}dm_{s}}+\int_{0}^{t}{h_{s-}dv_{s}}.$$
	But, $\int_{0}^{t}{h_{s-}dv_{s}}=\int_{0}^{t}{h_{s}dv_{s}}$ because $h$ is continuous. 
	\begin{enumerate}
	\item Hence, $\int_{0}^{\cdot}{h_{s}dM_{s}}\in\mathcal{M}(H)$ since $\int_{0}^{\cdot}{h_{s}dm_{s}}$ is a local martingale and it is obvious to see that $A=\int_{0}^{\cdot}{h_{s}dv_{s}}$ is a process with finite variations such that $dA$ is carried by $H$.
	\item Since $h$ vanishes on $H$ and $dv$ is carried by $H$, we obtain that $\int_{0}^{t}{h_{s}dv_{s}}=0$.  Consequently, $\int_{0}^{\cdot}{h_{s}dM_{s}}$  and $\int_{0}^{\cdot}{h_{s-}dM_{s}}$ are local martingales.
\end{enumerate}
\end{proof}

\begin{lem}\label{l2}
For any processes $M$ and $W$ of the class $\mathcal{M}(H)$, $MW-[ M,W]_{\cdot}$ is also an element of the class $\mathcal{M}(H)$.
\end{lem}
\begin{proof}
Through integration by parts, we have:
$$M_{t}W_{t}=\int_{0}^{t}{M_{s-}dW_{s}}+\int_{0}^{t}{W_{s-}dM_{s}}+[ M,W]_{t}.$$
Hence,
$$M_{t}W_{t}-[ M,W]_{t}=\int_{0}^{t}{M_{s-}dW_{s}}+\int_{0}^{t}{W_{s-}dM_{s}}.$$
Then, we obtain the result from Remark \ref{r1} and Lemma \ref{l1}.
\end{proof}

In what follows, we derive a series of corollaries of Lemma \ref{l2} which show that the process $MW-[ M,W]_{\cdot}$ can be a local martingale under some assumptions. 
\begin{coro}\label{c1}
If $M$ and $W$ are processes  of the class $\mathcal{M}(H)$ which vanish on $H$, hence $MW-[ M,W]_{\cdot}$ is a local martingale.
\end{coro}
\begin{proof}
Let $M$ and $W$ be processes of the class $\mathcal{M}(H)$. According to Lemma \ref{l2}, $MW-[ M,W]_{\cdot}$ is an element of the class $\mathcal{M}(H)$. Moreover, we have:
$$M_{t}W_{t}-[ M,W]_{t}=\int_{0}^{t}{M_{s-}dW_{s}}+\int_{0}^{t}{W_{s-}dM_{s}}.$$
But, we know from Lemma \ref{l1} that $\int_{0}^{t}{M_{s-}dW_{s}}$ and $\int_{0}^{t}{W_{s-}dM_{s}}$ are local martingales because $M$ and $W$ vanish on $H$. This completes the proof.
\end{proof}

\begin{coro}\label{c2}
If $M$ is a local martingale vanishing on $H$ and $W$ is a process  of the class $\mathcal{M}(H)$. Hence, $MW-[ M,W]_{\cdot}$ is a local martingale.
\end{coro}
\begin{proof}
According to Lemma \ref{l2}, $MW-[ M,W]_{\cdot}$ is also an element of the class $\mathcal{M}(H)$ and $\forall t\geq0$,
$$M_{t}W_{t}-[ M,W]_{t}=\int_{0}^{t}{M_{s-}dW_{s}}+\int_{0}^{t}{W_{s-}dM_{s}}.$$
We can remark that $\int_{0}^{\cdot}{W_{s-}dM_{s}}$ is a local martingale. Moreover, we deduce from Lemma \ref{l1} that $\int_{0}^{\cdot}{M_{s-}dW_{s}}$ is also a local martingale since $M$ vanishes on $H$.
\end{proof}

\begin{coro}\label{c3}
For any process $M$ of the class $\mathcal{M}(H)$, the process $MD-[ M,D]_{\cdot}$ is a local martingale.
\end{coro}
\begin{proof}
It is enough to notice that $M$ is a process of the class $\mathcal{M}(H)$ and $D$ is by definition, a martingale which vanishes on $H$. Thus, we obtain the result from Corollary \ref{c2}.
\end{proof}

In next corollary, we show that the product of the processes of class $\mathcal{M}(H)$ with vanishing quadratic covariations is again a relative martingale and in particular under some assumptions, a local martingale.
\begin{coro}\label{c4}
Let $(X_{t}^{1})_{t\geq0},\cdots,(X_{t}^{n})_{t\geq0}$ be processes of the class $\mathcal{M}(H)$ such that $[ X^{i},X^{j}]=0$ for $i\neq j$. Hence, the following hold:
\begin{enumerate}
	\item $(\Pi_{i=1}^{n}{X^{i}_{t}})_{t\geq0}$ is also of class $\mathcal{M}(H)$.
	\item If $\forall i\in\{1,\cdots, n\}$, $X^{i}$ vanishes on $H$. Hence, $(\Pi_{i=1}^{n}{X^{i}_{t}})_{t\geq0}$ is a local martingale.
	\item If $\exists l\in\{1,\cdots, n\}$ such that $X^{l}$ is a local martingale null on $H$. Hence, $(\Pi_{i=1}^{n}{X^{i}_{t}})_{t\geq0}$ is a local martingale.
\end{enumerate}
\end{coro}
\begin{proof}
\begin{enumerate}
	\item Let us first take $n=2$. Through Lemma \ref{l2}, we obtain that $X^{1}X^{2}-[X^{1},X^{2}]$ is a process of the class $\mathcal{M}(H)$. That is, $X^{1}X^{2}\in\mathcal{M}(H)$ since $[X^{1},X^{2}]=0$. Hence, we obtain by induction that for any family $(X_{t}^{1})_{t\geq0},\cdots,(X_{t}^{n})_{t\geq0}$ of the class $\mathcal{M}(H)$ such that $[ X^{i},X^{j}]=0$ for $i\neq j$, the process $(\Pi_{i=1}^{n}{X^{i}_{t}})_{t\geq0}$ is also of class $\mathcal{M}(H)$.
	\item We proceed in the same way as 1) by using Corollary \ref{c1} instead of Lemma \ref{l2} to show that $(\Pi_{i=1}^{n}{X^{i}_{t}})_{t\geq0}$ is a local martingale.
	\item Now, we assume that there exists $l\in\{1,\cdots,n\}$ such that $X^{l}$ is a local martingale null on $H$. Remark that:
	$$\Pi_{i=1}^{n}{X^{i}_{t}}=X^{l}\times\Pi_{i=1,i\neq l}^{n}{X^{i}_{t}}.$$
	But, we can see from 1) that $\Pi_{i=1,i\neq l}^{n}{X^{i}_{t}}\in\mathcal{M}(H)$. Hence, we obtain the result by using Corollary \ref{c2}.
\end{enumerate}
\end{proof}

Now, we shall derive the result from which Example \ref{ex2} follows.
\begin{lem}\label{l3}
Let $M$ and $W$ be processes of the class $\mathcal{M}(H)$ such that $W$ is continuous and $\{t\geq0:W_{t}=0\}\subset H$. Hence, next processes are elements of the class $\mathcal{M}(H)$:
\begin{enumerate}
	\item $X^{1}=\min{(M,M-W)}$;
	\item $X^{2}=\min{(M,M+W)}$;
	\item $X^{3}=\min{(M-W,M+W)}$;
	\item $X^{4}=\max{(M,M-W)}$;
  \item $X^{5}=\max{(M,M+W)}$;
	\item $X^{6}=\max{(M-W,M+W)}$.
\end{enumerate}
\end{lem}

\begin{proof}
Firstly, we obtain by using formulas $\min(x,y)=\frac{x+y-|x-y|}{2}$, $\max(x,y)=\frac{x+y+|x-y|}{2}$ that:
\begin{enumerate}
	\item $X^{1}_{t}=M_{t}-\frac{1}{2}W_{t}-\frac{1}{2}|W_{t}|$;
	\item $X^{2}=M_{t}+\frac{1}{2}W_{t}-\frac{1}{2}|W_{t}|$;
	\item $X^{3}=M_{t}-|W_{t}|$;
	\item $X^{4}=M_{t}-\frac{1}{2}W_{t}+\frac{1}{2}|W_{t}|$;
  \item $X^{5}=M_{t}+\frac{1}{2}W_{t}+\frac{1}{2}|W_{t}|$;
	\item $X^{6}=M_{t}+|W_{t}|$.
\end{enumerate}
Moreover, we have from Tanaka's formula that
$$|W_{t}|=\int_{0}^{t}{sign(W_{s})dW_{s}}+L_{t}^{0}(W).$$
We can see that $dL_{\cdot}^{0}(W)$ is carried by $H$ because $\{t\geq0:W_{t}=0\}\subset H$. Hence, $|W|\in\mathcal{M}(H)$ since according to Lemma \ref{l1}, $\int_{0}^{\cdot}{sign(W_{s})dW_{s}}\in\mathcal{M}(H)$. Consequently, we obtain from Remark \ref{r1} that the above mentioned processes are elements of the class $\mathcal{M}(H)$.
\end{proof}

Now, we shall derive some properties using the balayage formulas. Hence, we start by the predictable case.

\begin{lem}\label{l4}
Let $M$ be a continuous process of the class $\mathcal{M}(H)$, and let $g_{t}=\sup\{s\leq t:M_{s}=0\}$. Then, for any locally bounded predictable process $k$, $k_{g_{\cdot}}M$ is also an element of class $\mathcal{M}(H)$.
\end{lem}

\begin{proof}
By applying the balayage formula, we obtain the following:
$$k_{g_{t}}M_{t}=k_{g_{0}}M_{0}+\int_{0}^{t}{k_{g_{s}}dM_{s}}=\int_{0}^{t}{k_{g_{s}}dM_{s}}.$$
But, we know from Lemma \ref{l1} that $\int_{0}^{\cdot}{k_{g_{s}}dM_{s}}\in\mathcal{M}(H)$. This completes the proof.
\end{proof}

The following Corollary present us a situation under which relative martingales are also processes of the class $(\Sigma)$.
\begin{coro}\label{c5}
Any non-negative  process $M=m+v$ of the class $\mathcal{M}(H)$ which vanishes on $H$, is an element of the class $(\Sigma)$.
\end{coro}
\begin{proof}
Since $M$ vanishes on $H$, we obtain from Lemma \ref{l3} that for any locally bounded borel function $f$, $(f(v_{\gamma_{t}})M_{t}:t\geq0)\in\mathcal{M}(H)$, where $\gamma_{t}=\sup\{s\leq t:s\in H\}$. In addition, we have from balayage formula's that
$$f(v_{\gamma_{t}})M_{t}=\int_{0}^{t}{f(v_{\gamma_{s}})dM_{s}}.$$
But, $v_{\gamma_{t}}=v_{t}$ since $dv$ is carried by $H$. 
Then, 
$$f(v_{t})M_{t}=\int_{0}^{t}{f(v_{s})dm_{s}}+\int_{0}^{t}{f(v_{s})dv_{s}}.$$
Therefore, the process $\left(f(v_{t})M_{t}-\int_{0}^{t}{f(v_{s})dv_{s}}:t\geq0\right)$ is a local martingale. Consequently, we obtain from Theorem 2.4 of \cite{nik} that $M\in(\Sigma)$. This completes the Proof.
\end{proof}

\begin{rem}
In fact, all continuous relative martingales of the class $\mathcal{M}(H)$ which vanish on $H$, are processes of the class $(\Sigma)$. Indeed, it suffices to apply the above corollary to $|M|$ and to recall that $|M|\in(\Sigma)\Leftrightarrow M\in(\Sigma)$.
\end{rem}

Now, we shall use the balayage formula in progressive case to construct processes of the class $(\Sigma)$ from relative martingales.
\begin{lem}\label{l5}
Let $M$ be a continuous process of class $\mathcal{M}(H)$, and let $g_{t}=\sup\{s\leq t:M_{s}=0\}$. Then, for any càdlàg bounded progressive process $k$ which vanishes on $H$, $k_{g_{\cdot}}M$ is an element of class $(\Sigma)$.
\end{lem}

\begin{proof}
The balayage formula in progressive case through that $\forall t\geq0$,
$$k_{g_{t}}M_{t}=\int_{0}^{t}{^{p}(k_{g_{s}})dM_{s}}+R_{t},$$
where ${^{p}(k_{g_{s}})}$ is the predictable projection of $k_{g_{s}}$ and $R$ is a continuous process with finite variations such $dR$ is carried by $\{t\geq0:M_{t}=0\}$. Since $k$ is càdlàg, we have $^{p}(k_{g_{s}})=k_{s-}$. Hence, we obtain:
$$k_{g_{t}}M_{t}=\int_{0}^{t}{k_{s-}dM_{s}}+R_{t}$$
 Which implies the following:
$$k_{g_{t}}M_{t}=\int_{0}^{t}{k_{s}dM_{s}}+R_{t}$$
because $M$ is continuous. But, we have from Lemma \ref{l1} that $\int_{0}^{\cdot}{k_{s}dM_{s}}$ is a local martingale because $h$ vanishes on $H$. Consequently, the result holds. 
\end{proof}

\subsection{Relationship with the Azéma-Yor relative martingales}
Now, we shall state some relationship between the classes $\mathcal{M}(H)$ and $\mathcal{R}(H)$. More precisely, we derive some results which permit to decompose a process $M$ of the class $\mathcal{M}(H)$ as: 
\begin{equation}\label{e}
	M=M^{1}+M^{2},
\end{equation}
 where $M^{1}\in\mathcal{R}(H)$ and $M^{2}\in\mathcal{M}(H)$. Hence, we start by the following proposition:

\begin{prop}\label{p1}
Let $M$ be a process of the class $\mathcal{M}(H)$ such that its martingale part is uniformly integrable and $\langle M,D\rangle=0$. Hence, the process $(M_{t}-M_{\gamma_{t}}:t\geq0)$ is a relative martingale of the class $\mathcal{R}(H)$. 
\end{prop}
\begin{proof}
According to Corollary \ref{c3}, $DM-\langle D,M\rangle=DM$ is a uniformly integrable martingale. Hence, we obtain from quotient theorem that $(M_{t+\gamma})_{t\geq0}$ is a uniformly integrable martingale with respect to $(\mathcal{G}_{\gamma+t})_{t\geq0}$. Which entails that $(M_{t+\gamma}-M_{\gamma})_{t\geq0}$ is also a uniformly integrable martingale with respect to the filtration $(\mathcal{G}_{\gamma+t})_{t\geq0}$. Hence, there exists a random variable $M_{\infty}$ such that $M_{\gamma+t}-M_{\gamma}\to M_{\infty}$ and $\forall t>0$, we have:
$$M_{t+\gamma}-M_{\gamma}=\E[M_{\infty}|\mathcal{G}_{\gamma+t}].$$
But, we know thanks to Lemma 5.7 of \cite{jeu} that $\mathcal{G}_{\gamma+t}=\mathcal{F}_{\gamma+t}$. Then,
$$M_{t+\gamma}-M_{\gamma}=\E[M_{\infty}|\mathcal{F}_{\gamma+t}].$$
Hence, it follows that
$$\rho(M_{\cdot+\gamma}-M_{\gamma})_{t}=\rho(\E[M_{\infty}|\mathcal{F}_{\gamma+\cdot}])_{t}.$$
Now, let $Z_{t}=M_{t}-M_{\gamma_{t}}$ and $Z^{'}_{t}=\E[M_{\infty}1_{\{\gamma<t\}}|\mathcal{F}_{t}]$. We can remark that $Z$ and $Z^{'}$ vanish on $H$ and $\forall t\geq0$,
$$Z_{\gamma+t}=M_{t+\gamma}-M_{\gamma}\text{ and } Z^{'}_{t+\gamma}=\E[M_{\infty}|\mathcal{F}_{\gamma+t}].$$
Consequently, we obtain from uniqueness that 
$$M_{t}-M_{\gamma_{t}}=\E[M_{\infty}1_{\{\gamma<t\}}|\mathcal{F}_{t}].$$
Consequently, we conclude from Proposition 2.2 of \cite{1} that the process $(M_{t}-M_{\gamma_{t}}:t\geq0)$ is an element of the class $\mathcal{R}(H)$.
\end{proof}

\begin{rem}
We retain that under assumptions of Proposition \ref{p1}, there exists a random variable $M_{\infty}$ such that $M_{\gamma+t}-M_{\gamma}\to M_{\infty}$ as $t\rightarrow\infty$ and for any stopping time $0<T<\infty$, we have:
$$M_{T}-M_{\gamma_{T}}=\E[M_{\infty}1_{\{\gamma<T\}}|\mathcal{F}_{T}].$$
Hence, in the particular case where $M$ vanishes on $H$, we obtain the representation result given in Proposition 2.2 of \cite{1}. That is,
$$M_{T}=\E[M_{\infty}1_{\{\gamma<T\}}|\mathcal{F}_{T}].$$
\end{rem}

Next corollary permit us to see that under assumptions of Proposition \ref{p1}, the process $(M_{\gamma_{t}}:t\geq0)$ is also in the class $\mathcal{M}(H)$.
\begin{coro}\label{c6}
Let $M$ be a process of the class $\mathcal{M}(H)$ such that its martingale part is uniformly integrable and $\langle M,D\rangle=0$. Hence, the process $(M_{\gamma_{t}}:t\geq0)$ is a relative martingale of the class $\mathcal{M}(H)$.
\end{coro}
\begin{proof}
We have $\forall t\geq0$, $M_{\gamma_{t}}=(M_{\gamma_{t}}-M_{t})+M_{t}$. But, according to Proposition \ref{p1}, $(M_{t}-M_{\gamma_{t}}:t\geq0)\in\mathcal{R}(H)$. Hence, we obtain from Remark \ref{r1} that $(M_{\gamma_{t}}:t\geq0)\in\mathcal{M}(H)$.
\end{proof}

\begin{rem}
We obtain from Proposition \ref{p1} and Corollary \ref{c6} that any process $M$ satisfying assumptions of Proposition \ref{p1} admits the decomposition given in \eqref{e}, where $M^{1}=(M_{t}-M_{\gamma_{t}}:t\geq0)$ and $M^{2}=(M_{\gamma_{t}}:t\geq0)$. 
\end{rem}

In the following, we denote $\widetilde{M}$  to represent  the process defined by $\forall t\geq0$, 
$$\widetilde{M}_{t}=M_{\gamma+t}-M_{\gamma}-\int_{\gamma}^{\gamma+t}{\frac{d\langle M,D\rangle_{s}}{D_{s}}}.$$
Recall that from Theorem 4.1.2 of \cite{1}, $\widetilde{M}$ is a local martingale with respect to the filtration $(\mathcal{G}_{\gamma+t})_{t\geq0}$ when $M$ is a process of the class $\mathcal{M}(H)$. Hence, we derive an another decomposition of the form \eqref{e} for $M$ in the case where $\widetilde{M}$ is a true martingale.
\begin{prop}\label{p2}
Let $M$ be a process of the class $\mathcal{M}(H)$ such that $\widetilde{M}$ is a $(\mathcal{G}_{\gamma+t})_{t\geq0}$ uniformly integrable martingale. Hence, the process $\left(M_{t}-M_{\gamma_{t}}-\int_{\gamma_{t}}^{t}{\frac{d\langle M,D\rangle_{s}}{D_{s}}}:t\geq0\right)$ is an element of the class $\mathcal{R}(H)$. 
\end{prop}
\begin{proof}
Since $\widetilde{M}$ is a uniformly integrable martingale with respect to the filtration  $(\mathcal{G}_{\gamma+t})_{t\geq0}$. Hence, there exists an integrable  random variable $M_{\infty}$ such that $\widetilde{M}_{t}\to M_{\infty}$ as $t\rightarrow\infty$ and $\forall t\geq0$,
$$\widetilde{M}_{t}=\E[M_{\infty}|\mathcal{G}_{\gamma+t}]=\E[M_{\infty}|\mathcal{F}_{\gamma+t}].$$
Hence, it follows that
$$\rho(\widetilde{M})_{t}=\rho(\E[M_{\infty}|\mathcal{F}_{\gamma+\cdot}])_{t}.$$
Now, let 
$$Z_{t}=M_{t}-M_{\gamma_{t}}-\int_{\gamma_{t}}^{t}{\frac{d\langle M,D\rangle_{s}}{D_{s}}}\text{ and }Z^{'}_{t}=\E[M_{\infty}1_{\{\gamma<t\}}|\mathcal{F}_{t}].$$
 We can remark that $Z$ and $Z^{'}$ vanish on $H$ and $\forall t\geq0$,
$$Z_{\gamma+t}=\widetilde{M}_{t}\text{ and } Z^{'}_{t+\gamma}=\E[M_{\infty}|\mathcal{F}_{\gamma+t}].$$
Hence, we obtain that 
$$M_{t}-M_{\gamma_{t}}-\int_{\gamma_{t}}^{t}{\frac{d\langle M,D\rangle_{s}}{D_{s}}}=\E[M_{\infty}1_{\{\gamma<t\}}|\mathcal{F}_{t}].$$
Consequently, we conclude from Proposition 2.2 of \cite{1} that the process 
$$\left(M_{t}-M_{\gamma_{t}}-\int_{\gamma_{t}}^{t}{\frac{d\langle M,D\rangle_{s}}{D_{s}}}:t\geq0\right)$$
 is an element of the class $\mathcal{R}(H)$.
\end{proof}

\begin{coro}\label{c7}
Let $M$ be a process of the class $\mathcal{M}(H)$ such that $\widetilde{M}$ is a $(\mathcal{G}_{\gamma+t})_{t\geq0}$ uniformly integrable martingale. Hence, the process $\left(M_{\gamma_{t}}+\int_{\gamma_{t}}^{t}{\frac{d\langle M,D\rangle_{s}}{D_{s}}}:t\geq0\right)$ is also an element of the class $\mathcal{M}(H)$. 
\end{coro}
\begin{proof}
We have $\forall t\geq0$, 
$$M_{\gamma_{t}}+\int_{\gamma_{t}}^{t}{\frac{d\langle M,D\rangle_{s}}{D_{s}}}=\left(M_{\gamma_{t}}+\int_{\gamma_{t}}^{t}{\frac{d\langle M,D\rangle_{s}}{D_{s}}}-M_{t}\right)+M_{t}.$$ 
But, according to Proposition \ref{p2}, 
$$\left(M_{\gamma_{t}}+\int_{\gamma_{t}}^{t}{\frac{d\langle M,D\rangle_{s}}{D_{s}}}-M_{t}:t\geq0\right)\in\mathcal{R}(H).$$
 Hence, we obtain from Remark \ref{r1} that $\left(M_{\gamma_{t}}+\int_{\gamma_{t}}^{t}{\frac{d\langle M,D\rangle_{s}}{D_{s}}}:t\geq0\right)\in\mathcal{M}(H)$.
\end{proof}

\begin{coro}\label{c8}
Let $M$ be a process of the class $\mathcal{M}(H)$ such that $\widetilde{M}$ is a $(\mathcal{G}_{\gamma+t})_{t\geq0}$ uniformly integrable martingale and $d\langle M,D\rangle$ is carried by $H$. Hence,  the following hold: 
\begin{enumerate}
	\item $(M_{t}-M_{\gamma_{t}}:t\geq0)$ is an element of the class $\mathcal{R}(H)$;
	\item $(M_{\gamma_{t}}:t\geq0)$ is an element of the class $\mathcal{M}(H)$.
\end{enumerate} 
\end{coro}
\begin{proof}
We obtain respectively from Proposition \ref{p2} and Corollary \ref{c7} that
$$\left(M_{t}-M_{\gamma_{t}}-\int_{\gamma_{t}}^{t}{\frac{d\langle M,D\rangle_{s}}{D_{s}}}:t\geq0\right)\in\mathcal{R}(H)$$
and
$$\left(M_{\gamma_{t}}+\int_{\gamma_{t}}^{t}{\frac{d\langle M,D\rangle_{s}}{D_{s}}}:t\geq0\right)\in\mathcal{M}(H).$$
However, $\forall t\geq0$, $\int_{\gamma_{t}}^{t}{\frac{d\langle M,D\rangle_{s}}{D_{s}}}=0$ since $d\langle M,D\rangle$ is carried by $H$. Which completes the proof.
\end{proof}


\section{Applications to skew brownian motion equations}\label{sec:3}

In this section, weak solutions to time-homogeneous and time-inhomogeneous skew Brownian motions starting form zero are constructed on the one hand, with the help of a geometric Itô-Mckean skew Brownian motion with Azzalini skew normal distribution $X^{\delta}=\sqrt{1-\delta^{2}}B+\delta|W|$ and on the other hand, we do it from  arbitrary continuous processes of the class $\mathcal{M}(H)$. More precisely, we talk about of the two next equations:
\begin{equation}\label{1}
	X_{t}=x+B_{t}+(2\alpha-1)L_{t}^{0}(X)
\end{equation}
and
\begin{equation}\label{2}
	X_{t}=x+B_{t}+\int_{0}^{t}{(2\alpha(s)-1)dL_{s}^{0}(X)},
\end{equation}
where $B$ is a standard Brownian motion and $x=0$. It must be remark that solutions had already been built from the processes of the class $(\Sigma)$ (see \cite{fjo}). This should not be seen as a redundancy because the above mentioned processes are not necessary in the class $(\Sigma)$. Indeed, it is only when $X^{\delta}$ vanishes on $\{t\geq0:W_{t}=0\}$ that $X^{\delta}\in(\Sigma)$. And on another hand, an element $X$ of the class $\mathcal{M}(H)$ is in the class $(\Sigma)$ only when $X$ vanishes on $H$. 

\subsection{Construction of solution from Itô-Mckean skew brownian motion}

Recall that we presented $X^{\delta}$ in Section \ref{sec:1} as an element of the class $\mathcal{M}(H)$. In fact, this is true only when $W$ vanishes on $H$. In this subsection, we shall consider $X^{\delta}$ in general case. That is, $W$ does not necessarily vanish on $H$. Thus, under these assumptions, we construct from $X^{\delta}=\sqrt{1-\delta^{2}}B+\delta|W|$, solutions for Equations \ref{1} and \ref{2}. For this purpose, we shall set $Z^{W}$ and $Z^{1}$ to represent processes constructed in \eqref{zalpha} with respect to $W$ and $(Z^{W}_{g_{t}}X^{\delta}_{t};t\geq0)$ respectively and $Z^{2}$ is the process defined in \eqref{Zalpha} with respect to $(Z^{W}_{g_{t}}X^{\delta}_{t};t\geq0)$. We shall also set $g_{t}=\sup{\{t\geq0:X^{\delta}_{t}=0\}}$.
\begin{prop}
The process $Y^{\delta,1}$ defined by $\forall t\geq0$, $Y^{\delta,1}_{t}=Z^{1}_{t}Z^{W}_{g_{t}}X^{\delta}_{t}$ is a weak solution of \eqref{1} with the parameter $\alpha$ and starting from 0. 
\end{prop}
\begin{proof}
By applying the balayage formula in the progressive case, we get
$$Z^{W}_{g_{t}}X^{\delta}_{t}=\int_{0}^{t}{^{p}(Z^{W}_{g_{s}})dX^{\delta}_{s}}+R_{t},$$
where $R$ is a continuous process with finite variations such that $dR$ is carried by ${\{t\geq0:X^{\delta}_{t}=0\}}$. Since $W$ is continuous, we have: $^{p}(Z^{W}_{g_{s}})=Z^{W}_{g_{s-}}=Z^{W}_{s-}$. Thus, it follows from the continuity of $X^{\delta}$ that
$$Z^{W}_{g_{t}}X^{\delta}_{t}=\int_{0}^{t}{Z^{W}_{s}dX^{\delta}_{s}}+R_{t}.$$
Now, remark from Tanaka's formula that
$$dX^{\delta}_{s}=\sqrt{1-\delta^{2}}dB_{s}+\delta sign(W_{s})dW_{s}+\delta dL_{s}^{0}(W).$$
Hence, we obtain:
$$Z^{W}_{g_{t}}X^{\delta}_{t}=\sqrt{1-\delta^{2}}\int_{0}^{t}{Z^{W}_{s}dB_{s}}+\delta\int_{0}^{t}{Z^{W}_{s}sign(W_{s})dW_{s}}+\delta\int_{0}^{t}{Z^{W}_{s}dL_{s}^{0}(W)}+R_{t}.$$
Which becomes
$$Z^{W}_{g_{t}}X^{\delta}_{t}=\sqrt{1-\delta^{2}}\int_{0}^{t}{Z^{W}_{s}dB_{s}}+\delta\int_{0}^{t}{Z^{W}_{s}sign(W_{s})dW_{s}}+R_{t}$$
since $dL^{0}(W)$ is carried by ${\{t\geq0:W_{t}=0\}}={\{t\geq0:Z^{W}_{t}=0\}}$. Hence, through Proposition \ref{pzalph}, we get
 $$Y^{\delta,1}_{t}=\sqrt{1-\delta^{2}}\int_{0}^{t}{Z^{1}_{s}Z^{W}_{s}dB_{s}}+\delta\int_{0}^{t}{Z^{1}_{s}Z^{W}_{s}sign(W_{s})dW_{s}}+\int_{0}^{t}{Z^{1}_{s}dR_{s}}+(2\alpha-1)L_{0}^{0}(Y^{\delta,1}).$$
But, $dR$ is carried by ${\{t\geq0:X^{\delta}_{t}=0\}}$ and $\{t\geq0:X^{\delta}_{t}=0\}\subset\{t\geq0:Z^{W}_{g_{t}}X^{\delta}_{t}=0\}=\{t\geq0:Z^{1}_{t}=0\}$. Therefore,
$$Y^{\delta,1}_{t}=\sqrt{1-\delta^{2}}\int_{0}^{t}{Z^{1}_{s}Z^{W}_{s}dB_{s}}+\delta\int_{0}^{t}{Z^{1}_{s}Z^{W}_{s}sign(W_{s})dW_{s}}+(2\alpha-1)L_{t}^{0}(Y^{\delta,1}).$$
Now, remark that the process $M$ defined by $\forall t\geq0$, 
$$M_{t}=\sqrt{1-\delta^{2}}\int_{0}^{t}{Z^{1}_{s}Z^{W}_{s}dB_{s}}+\delta\int_{0}^{t}{Z^{1}_{s}Z^{W}_{s}sign(W_{s})dW_{s}}$$
 is a continuous local martingale. In addition, we have thanks to the continuity of processes $B$ and $W$:
$$M_{t}=\sqrt{1-\delta^{2}}\int_{0}^{t}{k^{1}_{s}k^{W}_{s}dB_{s}}+\delta\int_{0}^{t}{k^{1}_{s}k^{W}_{s}sign(W_{s})dW_{s}},$$
where $k^{1}$ and $k^{W}$ are progressive processes constructed in \eqref{zalpha} with respect to $W$ and $(Z^{W}_{g_{t}}X^{\delta}_{t};t\geq0)$ respectively. On another hand, we have:
$$\langle M,M\rangle_{t}=(1-\delta^{2})\int_{0}^{t}{(k^{1}_{s}k^{W}_{s})^{2}ds}+\delta^{2}\int_{0}^{t}{(k^{1}_{s}k^{W}_{s}sign(W_{s}))^{2}ds}.$$
Which implies: $\langle M,M\rangle_{t}=t$ because $k^{1}_{s}\in\{-1,1\}$, $k^{W}_{s}\in\{-1,1\}$ and $sign(W_{s})\in\{-1,1\}$. Consequently, $M$ is a Brownian motion. This completes the proof.
\end{proof}

\begin{prop}
The process $Y^{\delta,2}$ defined by $\forall t\geq0$, $Y^{\delta,2}_{t}=Z^{2}_{t}Z^{W}_{g_{t}}X^{\delta}_{t}$ is a weak solution of \eqref{2} with the parameter $\alpha$ and starting from 0. 
\end{prop}
\begin{proof}
We have yet showed in the above last proof that
$$Z^{W}_{g_{t}}X^{\delta}_{t}=\sqrt{1-\delta^{2}}\int_{0}^{t}{Z^{W}_{s}dB_{s}}+\delta\int_{0}^{t}{Z^{W}_{s}sign(W_{s})dW_{s}}+R_{t}.$$
Hence, from Proposition \ref{siam}, we get
 $$Y^{\delta,2}_{t}=\sqrt{1-\delta^{2}}\int_{0}^{t}{Z^{2}_{s}Z^{W}_{s}dB_{s}}+\delta\int_{0}^{t}{Z^{2}_{s}Z^{W}_{s}sign(W_{s})dW_{s}}+\int_{0}^{t}{Z^{2}_{s}dR_{s}}+\int_{0}^{t}{(2\alpha(s)-1)dL_{s}^{0}(Y^{\delta,2})}.$$
But, $dR$ is carried by ${\{t\geq0:X^{\delta}_{t}=0\}}$ and $\{t\geq0:X^{\delta}_{t}=0\}\subset\{t\geq0:Z^{W}_{g_{t}}X^{\delta}_{t}=0\}=\{t\geq0:Z^{2}_{t}=0\}$. Therefore,
$$Y^{\delta,2}_{t}=\sqrt{1-\delta^{2}}\int_{0}^{t}{Z^{2}_{s}Z^{W}_{s}dB_{s}}+\delta\int_{0}^{t}{Z^{2}_{s}Z^{W}_{s}sign(W_{s})dW_{s}}+\int_{0}^{t}{(2\alpha(s)-1)dL_{s}^{0}(Y^{\delta,2})}.$$
Now, remark that the process $M$ defined by $\forall t\geq0$, 
$$M^{'}_{t}=\sqrt{1-\delta^{2}}\int_{0}^{t}{Z^{2}_{s}Z^{W}_{s}dB_{s}}+\delta\int_{0}^{t}{Z^{2}_{s}Z^{W}_{s}sign(W_{s})dW_{s}}$$
 is a continuous local martingale. In addition, we have thanks to the continuity of processes $B$ and $W$:
$$M^{'}_{t}=\sqrt{1-\delta^{2}}\int_{0}^{t}{k^{2}_{s}k^{W}_{s}dB_{s}}+\delta\int_{0}^{t}{k^{1}_{s}k^{W}_{s}sign(W_{s})dW_{s}},$$
where $k^{2}$ and $k^{W}$ are progressive processes constructed in () with respect to $W$ and $(Z^{W}_{g_{t}}X^{\delta}_{t};t\geq0)$ respectively. On another hand, we have:
$$\langle M^{'},M^{'}\rangle_{t}=(1-\delta^{2})\int_{0}^{t}{(k^{2}_{s}k^{W}_{s})^{2}ds}+\delta^{2}\int_{0}^{t}{(k^{2}_{s}k^{W}_{s}sign(W_{s}))^{2}ds}.$$
Which implies: $\langle M^{'},M^{'}\rangle_{t}=t$ because $k^{2}_{s}\in\{-1,1\}$, $k^{W}_{s}\in\{-1,1\}$ and $sign(W_{s})\in\{-1,1\}$. Consequently, $M^{'}$ is a Brownian motion. This completes the proof.
\end{proof}

\subsection{Construction of solutions from relative martingales}
Now, we shall derive solutions by using continuous processes of the class $\mathcal{M}(H)$. Thus, for any continuous process $X$ of the last mentioned class, we let $g_{t}=\sup\{s\leq t:X_{t}=0\}$ and $\tau_{t}=\inf\{s\geq0:\langle X,X\rangle_{s}>t\}$. Let $Z^{D}$ and $Z^{1}$ be progressive processes defined in \eqref{zalpha} with respect to $D$ and $(Z^{D}_{g_{t}}X_{t}:t\geq0)$ respectively. $Z^{2}$ is the progressive process defined in \eqref{Zalpha} with respect to $(Z^{D}_{g_{t}}X_{t}:t\geq0)$.
\begin{prop}
The process $\mathcal{Y}^{1}$ defined by $\forall t\geq0$, $\mathcal{Y}^{1}_{t}=Z^{1}_{t}Z^{D}_{g_{\tau_{t}}}X_{\tau_{t}}$ is a weak solution of \eqref{1} with the parameter $\alpha$ and starting from 0. 
\end{prop}
\begin{proof}
First remark that $Z^{D}$ is a continuous bounded progressive process which vanishes on $H$. Hence, we obtain from Lemma \ref{l6} that $(Z^{D}_{g_{t}}X_{t}:t\geq0)$ is a continuous process of the class $(\Sigma)$. Hence, we obtain the result by applying Proposition 8 of \cite{fjo} on the process  $(Z^{D}_{g_{t}}X_{t}:t\geq0)$.
\end{proof}

\begin{prop}
The process $\mathcal{Y}^{2}$ defined by $\forall t\geq0$, $\mathcal{Y}^{2}_{t}=Z^{2}_{t}Z^{D}_{g_{\tau_{t}}}X_{\tau_{t}}$ is a weak solution of \eqref{2} with the parameter $\alpha$ and starting from 0. 
\end{prop}
\begin{proof}
We obtain the result by applying Proposition 9 of \cite{fjo} on the process  $(Z^{D}_{g_{t}}X_{t}:t\geq0)$.

\end{proof}


\section{Stochastic differential equations driven by a relative martingale}\label{sec:4}

In this section, we study stochastic differential equations driven by a relative martingale. More precisely, we investigate stochastic differential equations of the form:
\begin{equation}\label{eds}
	dX_{t}=\sigma(t,X_{t})dW_{t}+b(t,X_{t})dt\text{, }0\leq t\leq T\text{, }X_{0}=Z
\end{equation}
where $W=B+v$ is a continuous sub-martingale of the class $\mathcal{M}(H)$ such that $B$ is a standard Brownian motion. The study of such equations can have good applications in finance engineering. For instance, one of such equations has recently appeared in \cite{ruin} to propose an investment strategy for insurance companies. Specifically, that is next equation:
$$dS_{t}=\left(\mu+\frac{\sigma^{2}}{2}\right)S_{t}dt+\sigma S_{t}dX^{\delta}_{t},$$
where $X^{\delta}$ is the Itô-Mckean skew Brownian motion presented in Section \ref{sec:2} as a process of the class $\mathcal{M}(H)$. In particular, the present investigations will be do under next hypothesis:

\begin{hyp}\label{h2}
Let $T>0$ and $b:[0,T]\times\R\rightarrow\R$, $\sigma:[0,T]\times\R\rightarrow\R$ be measurable functions satisfying
\begin{equation}\label{H1}
	|b(t,x)|+|\sigma(t,x)|\leq C(1+|x|)\text{; }t\in[0,T]\text{, }x\in\R
\end{equation}
for some constant $C$ and such that
\begin{equation}\label{H2}
	|b(t,x)-b(t,y)|+|\sigma(t,x)-\sigma(t,y)|\leq K|x-y|\text{; }t\in[0,T]\text{, }x,y\in\R
\end{equation}
for some constant $K$. Let $Z$ be a random variable which is independent of the $\sigma-$ algebra $\mathcal{F}_{\infty}$ generated by $B_{t}$, $t\geq0$ and such that 
$$\E[|Z|^{2}]<\infty\text{ and }\E[V_{T}|Z|^{2}]<\infty.$$
\end{hyp}

Under the above hypothesis, the classical stochastic equation:
\begin{equation}\label{ceds}
	dY_{t}=\sigma(t,Y_{t})dB_{t}+b(t,Y_{t})dt\text{, }Y_{0}=Z
\end{equation}
admits a unique continuous solution (see Theorem 5.2.1 of \cite{bo}). Throughout the rest of this paper, we shall denote this solution $Y$.
The study of Equation \eqref{eds} strongly depends on the random set $H$. Indeed, the novelty in this equation comes from the integral $\int_{0}^{t}{\sigma(s,X_{s})dv_{s}}$ whose behaviour depends on $H$ since $dv$ is carried by $H$. Hence, the present section consists of two principal subsections. In the first one, we investigate \eqref{eds} according to the structure of $H$. In the second part, we approach the study in a more general way without taking into account the structure on $H$.
\subsection{Relationship with the classical equation}
We first remark that $dW=dB$ outside set $H$. Hence, under some conditions, the solution $X$ of \eqref{eds} behaves like the solution $Y$ of \eqref{ceds}. In this subsection, we investigate situations where the solution $Y$ of \eqref{ceds} satisfies \eqref{eds}. Thus, we start by show that $Y$ is also solution of \eqref{eds} when $t\mapsto\sigma(t,x)$ vanishes on $H$.

\begin{prop}
If in addition to Hypothesis \ref{h2}, the function $\sigma$ is such that $\forall s\in H$ and $\forall x\in \R$, $\sigma(s,x)=0$. Hence, the unique solution of the equation:
$$X_{t}=Z+\int_{0}^{t}{\sigma(s,X_{s})dB_{s}}+\int_{0}^{t}{b(s,X_{s})ds}$$
is also the unique solution of the next equation:
$$X_{t}=Z+\int_{0}^{t}{\sigma(s,X_{s})dW_{s}}+\int_{0}^{t}{b(s,X_{s})ds}.$$
\end{prop}
\begin{proof}
Firstly, we have: $\forall t\geq0$,
$$\int_{0}^{t}{\sigma(s,X_{s})dW_{s}}=\int_{0}^{t}{\sigma(s,X_{s})dB_{s}}+\int_{0}^{t}{\sigma(s,X_{s})dv_{s}}.$$
But, $\int_{0}^{t}{\sigma(s,X_{s})dv_{s}}=0$ because, $s\mapsto\sigma(s,X_{s})$ vanishes on $H$ and $dv$ is carried by $H$. Which proves that the two above equations are equivalent. This completes the proof.
\end{proof}

Now, recall that $H$ is the zero set of a continuous martingale $D$. Hence, $H$ cannot be ordered. However $\R_{+}\setminus H$ can be decomposed as a countable union $\cup_{n\in\N}{J_{n}}$ of intervals $J_{n}$. Each interval $J_{n}$ corresponds to some excursion of $D$. Specifically, if $J_{n}=]g_{n},d_{n}[$, $dW_{t}=dB_{t}$ for all $t\in]g_{n},d_{n}[$ and $g_{n},d_{n}\in H$. In the following, we explore situations where the solution $Y$ of \eqref{ceds} satisfies \eqref{eds}. Let $\tau_{1}=g_{0}$ be the first zero of $D$ and denote $N=\inf\{n\geq0:d_{n}\neq g_{n+1}\}$ and $\tau=d_{N}$.

In the following, we show that Equation \eqref{eds} admits a unique solution before the first entry time in $H$, $\tau_{1}$. And that this solution is the same which verifies \eqref{ceds}.

\begin{prop}\label{h}
Let $T$ be a real such that $T<\tau_{1}$ a.s. Under Assumptions \ref{h2}, there exists a unique continuous process $X$ such that $\forall t\in[0,T]$,
$$X_{t}=Z+\int_{0}^{t}{\sigma(s,X_{s})dW_{s}}+\int_{0}^{t}{b(s,X_{s})ds}.$$
It is the same solution of \eqref{ceds}.
\end{prop}
\begin{proof}
First remark that $v$ is constant on $[0,T]$ because, $dv$ is carried by $H$ and $[0,T]\subset[0,\tau_{1}[\subset H^{c}$. That is, we have: $dW_{s}=dB_{s}$, $\forall s\leq T$. Hence, \eqref{eds} coincides with the following standard stochastic differential equation:
\begin{equation}\label{seds}
	X_{t}=Z+\int_{0}^{t}{\sigma(s,X_{s})dB_{s}}+\int_{0}^{t}{b(s,X_{s})ds}.
\end{equation}
Consequently, we obtain the existence from Theorem 5.2.1 of \cite{bo}. Which completes the proof.
\end{proof}

In the next proposition, we show that the above result is again true on $[0,\tau_{2}[$, where $\tau_{2}=\inf\{t>\tau_{1}: t\in H\}$.
\begin{prop}
For all $T>0$ such that $\gamma_{T}=\tau_{1}$, we have under Assumption \ref{h2}, that there exists a unique continuous process $X$ such that $\forall t\leq T$,
$$X_{t}=Z+\int_{0}^{t}{\sigma(s,X_{s})dW_{s}}+\int_{0}^{t}{b(s,X_{s})ds}.$$
\end{prop}
\begin{proof} 
We know from Theorem 5.2.1 of \cite{bo} that the classic Equation: 
$$X_{t}=Z+\int_{0}^{t}{\sigma(s,X_{s})dB_{s}}+\int_{0}^{t}{b(s,X_{s})ds}$$
admits a unique continuous solution $X$. According to Proposition \ref{h}, $X$ is also a solution of \eqref{eds} on $[0,\tau_{1}[$. Furthermore, $\forall t\in[\tau_{1},T]$,
$$X_{t}=Z+\int_{0}^{\tau_{1}}{\sigma(s,X_{s})dB_{s}}+\int_{\tau_{1}}^{t}{\sigma(s,X_{s})dB_{s}}+\int_{0}^{t}{b(s,X_{s})ds}.$$
But,$\forall t\in[\tau_{1},T]$, $\gamma_{t}=\gamma_{T}$ because $\gamma_{T}=\tau_{1}$. Hence, we get:
$$X_{t}=Z+\int_{0}^{\tau_{1}}{\sigma(s,X_{s})dB_{s}}+\int_{\gamma_{t}}^{t}{\sigma(s,X_{s})dB_{s}}+\int_{0}^{t}{b(s,X_{s})ds}.$$
However, $dB=dW$ on $[0,\tau_{1}[$ and on $[\gamma_{t},t]$. Which implies:
$$X_{t}=Z+\int_{0}^{\tau_{1}}{\sigma(s,X_{s})dW_{s}}+\int_{\gamma_{t}}^{t}{\sigma(s,X_{s})dW_{s}}+\int_{0}^{t}{b(s,X_{s})ds}.$$
That is,
$$X_{t}=Z+\int_{0}^{t}{\sigma(s,X_{s})dW_{s}}+\int_{0}^{t}{b(s,X_{s})ds}.$$
Consequently, $X$ is also solution of \eqref{eds}.
\end{proof}

Now, we show that Equation \eqref{eds} admits a unique solution on $[0,\tau]$ and that this solution is also the same which verifies \eqref{ceds}.
\begin{prop}
Let $T$ be a real such that $T\leq\tau$ a.s. If in addition, $\tau_{1}<\tau$. Hence, under Assumptions \ref{h2}, there exists a unique continuous process $X$ such that $\forall t\in[0,T]$,
$$X_{t}=Z+\int_{0}^{t}{\sigma(s,X_{s})dW_{s}}+\int_{0}^{t}{b(s,X_{s})ds}.$$
It is the same solution of \eqref{ceds}.
\end{prop}
\begin{proof}
First remark that $\forall T\leq\tau$, $[0,T]\cap H$ is a finite and countable set. That is, $\forall t\leq T$, there exist an integer $d$ and reals, $t_{1},\cdots,t_{d}$ such that $[0,t]\cap H=\{t_{1},\cdots,t_{d}\}$, where $\tau_{1}=t_{1}<t_{2}<\cdots<t_{d}=\gamma_{t}$. Thus, we have:
$$\int_{0}^{t}{\sigma(s,X_{s})dW_{s}}=\int_{0}^{t_{1}}{\sigma(s,X_{s})dW_{s}}+\sum_{k=1}^{d-1}{\int_{t_{k}}^{t_{k+1}}{\sigma(s,X_{s})dW_{s}}}+\int_{\gamma_{t}}^{t}{\sigma(s,X_{s})dW_{s}}.$$
But, $v$ is constant on $[0,t_{1}[$, $[\gamma_{t},t]$ and on $[t_{k},t_{k+1}[$, $\forall k\in\{1,\cdots,d-1\}$ because $D$ does not vanish on intervals $[0,t_{1}[$, $]\gamma_{t},t]$ and on $]t_{k},t_{k+1}[$. Hence,
$$\int_{0}^{t}{\sigma(s,X_{s})dW_{s}}=\int_{0}^{t_{1}}{\sigma(s,X_{s})dB_{s}}+\sum_{k=1}^{d-1}{\int_{t_{k}}^{t_{k+1}}{\sigma(s,X_{s})dB_{s}}}+\int_{\gamma_{t}}^{t}{\sigma(s,X_{s})dB_{s}}=\int_{0}^{t}{\sigma(s,X_{s})dB_{s}}.$$
Which means that the equation
$$X_{t}=Z+\int_{0}^{t}{\sigma(s,X_{s})dW_{s}}+\int_{0}^{t}{b(s,X_{s})ds}$$
is equivalent to
$$X_{t}=Z+\int_{0}^{t}{\sigma(s,X_{s})dB_{s}}+\int_{0}^{t}{b(s,X_{s})ds}.$$
This completes the proof.
\end{proof}

Now, we shall explore what happens after the honest time $\gamma=\sup\{t\geq0:t\in H\}$. In particular, we show that we have the previous result in the enlarged filtration $(\mathcal{G}_{\gamma+t})_{t\geq0}$.
\begin{prop}\label{p10}
Under Assumptions \ref{h2}, there exists a unique continuous process $Y$, adapted to the filtration \\$(\mathcal{G}_{\gamma+t})_{t\geq0}$ such that $\forall t\geq0$,
$$Y_{t}=Z+\int_{0}^{t}{\sigma(s,Y_{s})d\widetilde{W}_{s}}+\int_{0}^{t}{b(s,Y_{s})ds},$$
where $\widetilde{W}_{t}=W_{\gamma+t}-W_{\gamma}-\int_{\gamma}^{\gamma+t}{\frac{d\langle W,D\rangle_{s}}{D_{s}}}$.
\end{prop}
\begin{proof}
We obtain from Theorem 2.4.1 of \cite{1} that $\left(W_{\gamma+t}-W_{\gamma}-\int_{\gamma}^{\gamma+t}{\frac{d\langle W,|D|\rangle_{s}}{|D_{s}|}}:t\geq0\right)$ is a local martingale with respect to $(\mathcal{G}_{\gamma+t})_{t\geq0}$. But, we have:
$$\int_{\gamma}^{\gamma+t}{\frac{d\langle W,|D|\rangle_{s}}{|D_{s}|}}=\int_{\gamma}^{\gamma+t}{\frac{d\langle W,D\rangle_{s}}{D_{s}}}.$$
Hence, $\widetilde{W}$ is a local martingale adapted to $(\mathcal{G}_{\gamma+t})_{t\geq0}$. Moreover,
$$\langle \widetilde{W},\widetilde{W}\rangle_{t}=\langle W,W\rangle_{\gamma+t}-\langle W,W\rangle_{\gamma}=t.$$
Then, $\widetilde{W}$ is a Brownian motion with respect to $(\mathcal{G}_{\gamma+t})_{t\geq0}$. Consequently, we obtain from Theorem 5.2.1 of \cite{bo} that there exists a unique continuous process $Y$ adapted to filtration $(\mathcal{G}_{\gamma+t})_{t\geq0}$ such that
$$Y_{t}=Z+\int_{0}^{t}{\sigma(s,Y_{s})d\widetilde{W}_{s}}+\int_{0}^{t}{b(s,Y_{s})ds}.$$
\end{proof}

\begin{coro}\label{l6}
Let $W=B+v$ be a $(\mathcal{F}_{t})_{t\geq0}$- adapted and continuous process  of the class $\mathcal{M}(H)$ such that $B$ is a standard Brownian motion and $d\langle W,D\rangle$ is carried by $H$. Under Hypothesis \ref{h2}, there exists a process $X$ such that $\forall t\geq0$,
\begin{equation}\label{eeq1}
	X_{t}=X_{\gamma_{t}}+\int_{\gamma_{t}}^{t}{\sigma(X_{s})dW_{s}}+\int_{\gamma_{t}}^{t}{b(X_{s})ds}.
\end{equation}
\end{coro}
\begin{proof}
We first remark that $\widetilde{W}_{t}=W_{t+\gamma}-W_{\gamma}$ since $\int_{\gamma}^{\gamma+t}{\frac{d\langle W,D\rangle_{s}}{D_{s}}}=0$. Indeed, $d\langle W,D\rangle$ is carried by $H$. Moreover, we know from Proposition \ref{p10} that there exists a unique continuous process $Y$ adapted to $(\mathcal{G}_{\gamma+t})_{t\geq0}$ such that $\forall t\geq0$,
$$Y_{t}=Y_{0}+\int_{0}^{t}{\sigma(Y_{s})d\widetilde{W}_{s}}+\int_{0}^{t}{b(Y_{s})ds}.$$
On another hand, we get from Chapter V \cite{jeu} that there exists a process $X$ adapted to $(\mathcal{F}_{t})_{t\geq0}$ such that $\forall t\geq0$, $Y_{t}=X_{\gamma+t}$. Hence, we obtain:
$$X_{\gamma+t}=X_{\gamma}+\int_{0}^{t}{\sigma(X_{\gamma+s})d\widetilde{W}_{s}}+\int_{0}^{t}{b(X_{\gamma+s})ds}.$$
Then, it follows:
$$X_{\gamma+t}-X_{\gamma}=\int_{\gamma}^{\gamma+t}{\sigma(X_{s})dW_{s}}+\int_{\gamma}^{\gamma+t}{b(X_{s})ds}.$$
Which implies that
$$\rho(X_{\gamma+\cdot}-X_{\gamma})_{t}=\rho\left(\int_{\gamma}^{\gamma+\cdot}{\sigma(X_{s})dW_{s}}+\int_{\gamma}^{\gamma+\cdot}{b(X_{s})ds}\right)_{t}.$$
Now, let us consider processes $Z$ and $Z^{'}$ defined by $\forall t\geq0$,
$Z_{t}=X_{t}-X_{\gamma_{t}}$ and $Z^{'}_{t}=\int_{\gamma_{t}}^{t}{\sigma(X_{s})dW_{s}}+\int_{\gamma_{t}}^{t}{b(X_{s})ds}$. We can see that $Z$ and $Z^{'}$ vanish on $H$ and that $\forall t\geq0$, 
$$Z_{\gamma+t}=X_{\gamma+t}-X_{\gamma}\text{ and }Z_{\gamma+t}^{'}=\int_{\gamma}^{\gamma+t}{\sigma(X_{s})dW_{s}}+\int_{\gamma}^{\gamma+t}{b(X_{s})ds}.$$ Consequently, we obtain by uniqueness that
$$X_{t}-X_{\gamma_{t}}=\int_{\gamma_{t}}^{t}{\sigma(X_{s})dW_{s}}+\int_{\gamma_{t}}^{t}{b(X_{s})ds}.$$
This proves the existence of solutions for \eqref{eeq1}.
\end{proof}

Note that any solution of the equation:
\begin{equation}\label{eq1}
	X_{t}=X_{0}+\int_{0}^{t}{\sigma(X_{s})dW_{s}}+\int_{0}^{t}{b(X_{s})ds},
\end{equation}
verifies also \eqref{eeq1}. But, we cannot always affirm the reciprocal.  In next corollary, we give a sufficient condition under which a solution of \eqref{eeq1} is also a solution of \eqref{eq1}.

\begin{coro}\label{c9}
Let $X$ be a solution of Equation \eqref{eeq1} such that $\forall t\in H$, $X_{t}=X_{0}+\int_{0}^{t}{\sigma(X_{s})dW_{s}}+\int_{0}^{t}{b(X_{s})ds}$. Hence, $X$ is a solution of \eqref{eq1} for every $t\geq0$.
\end{coro}
\begin{proof}
We have: $\forall t\geq0$,
$$X_{t}=X_{\gamma_{t}}+\int_{\gamma_{t}}^{t}{\sigma(X_{s})dW_{s}}+\int_{\gamma_{t}}^{t}{b(X_{s})ds}$$
since $X$ is solution of \eqref{eeq1}. In addition, $\forall t\in H$, 
$$X_{t}=X_{0}+\int_{0}^{t}{\sigma(X_{s})dW_{s}}+\int_{0}^{t}{b(X_{s})ds}.$$
Which means that
$$X_{\gamma_{t}}=X_{0}+\int_{0}^{\gamma_{t}}{\sigma(X_{s})dW_{s}}+\int_{0}^{\gamma_{t}}{b(X_{s})ds}.$$
This implies that $\forall t\geq0$,
$$X_{t}=X_{0}+\int_{0}^{\gamma_{t}}{\sigma(X_{s})dW_{s}}+\int_{\gamma_{t}}^{t}{\sigma(X_{s})dW_{s}}+\int_{0}^{\gamma_{t}}{b(X_{s})ds}+\int_{\gamma_{t}}^{t}{b(X_{s})ds}.$$
Consequently, we obtain that $\forall t\geq0$,
$$X_{t}=X_{0}+\int_{0}^{t}{\sigma(X_{s})dW_{s}}+\int_{0}^{t}{b(X_{s})ds}.$$
This completes the proof.
\end{proof}

\subsection{Study in the general case}

Now, we shall investigate next equation:
\begin{equation}\label{eds1}
	X_{t}=\zeta1_{H}(t)+\left[Z+\int_{0}^{t}{\sigma(s,X_{s})dW_{s}}+\int_{0}^{t}{b(s,X_{s})ds}\right]1_{H^{c}}(t),
\end{equation}
where $W=B+v$ is a sub-martingale of the class $\mathcal{M}(H)$ such that $B$ is a standard Brownian motion. Thus, we shall consider next assumptions:
 
We prove existence of solutions for \eqref{eds1} in what follows.
\begin{prop}
Under the above assumptions, there exist solutions for \eqref{eds1}.
\end{prop}
\begin{proof}
Define $Y^{(0)}=X_{0}$ and $Y^{(p)}=Y^{(p)}(w)$ inductively as follows
$$Y^{(p+1)}_{t}=\zeta1_{H}(t)+\left[Z+\int_{0}^{t}{\sigma(s,Y^{(p)}_{s})dW_{s}}+\int_{0}^{t}{b(s,Y^{(p)}_{s})ds}\right]1_{H^{c}}(t).$$
We have $\forall t\leq T$ and $p\geq1$,
$$\E\left[|Y^{(p+1)}_{t}-Y^{(p)}_{t}|^{2}\right]\leq2\E\left[\int_{0}^{t}(\sigma(s,Y^{(p)}_{s})-\sigma(s,Y^{(p-1)}_{s}))dW_{s}\right]^{2}+2\E\left[\int_{0}^{t}(b(s,Y^{(p)}_{s})-b(s,Y^{(p-1)}_{s}))ds\right]^{2}.$$
Let $\alpha_{s}=\sigma(s,Y^{(p)}_{s})-\sigma(s,Y^{(p-1)}_{s})$ and $\beta_{s}=b(s,Y^{(p)}_{s})-b(s,Y^{(p-1)}_{s})$. Thus, we obtain:
$$\E\left[\int_{0}^{t}{\beta_{s}ds}\right]^{2}\leq K^{2}t\int_{0}^{t}{\E[|Y^{(p)}_{s}-Y^{(p-1)}_{s}|^{2}]ds}.$$
And,
$$\E\left[\int_{0}^{t}{\alpha_{s}dW_{s}}\right]^{2}\leq2\E\left[\int_{0}^{t}{\alpha_{s}dB_{s}}\right]^{2}+2\E\left[\int_{0}^{t}{\alpha_{s}dv_{s}}\right]^{2}$$
$$\hspace{2.5cm}\leq2\int_{0}^{t}{\E(\alpha_{s})^{2}ds}+2\E\left[\int_{0}^{t}{|\alpha_{s}|dv_{s}}\right]^{2}.$$
Hence, we get by assumptions the following:
$$\E\left[\int_{0}^{t}{\alpha_{s}dW_{s}}\right]^{2}\leq2K^{2}\int_{0}^{t}{\E[|Y^{(p)}_{s}-Y^{(p-1)}_{s}|^{2}]ds}+2K^{2}\E\left[\int_{0}^{t}{|Y^{(p)}_{s}-Y^{(p-1)}_{s}|dv_{s}}\right]^{2}.$$
But $\forall s\in H$, $\int_{0}^{t}{|Y^{(p)}_{s}-Y^{(p-1)}_{s}|dv_{s}}=0$ since $Y^{(p)}_{s}-Y^{(p-1)}_{s}=0$. 
Which implies:
$$\E\left[\int_{0}^{t}{\alpha_{s}dW_{s}}\right]^{2}\leq2K^{2}\int_{0}^{t}{\E[|Y^{(p)}_{s}-Y^{(p-1)}_{s}|^{2}]ds}.$$
Therefore, we get:
$$\E\left[|Y^{(p+1)}_{t}-Y^{(p)}_{t}|^{2}\right]\leq2K^{2}(2+t)\int_{0}^{t}{\E[|Y^{(p)}_{s}-Y^{(p-1)}_{s}|^{2}]ds}.$$
In addition, we have:
$$\hspace{-4cm}\E\left[|Y^{(1)}_{t}-Y^{(0)}_{t}|^{2}\right]\leq\E\left[\left|\zeta+\int_{0}^{t}{\sigma(s,X_{0})dW_{s}}+\int_{0}^{t}{b(s,X_{0})ds}\right|^{2}\right]$$
$$\hspace{2cm}\leq3\E[\zeta^{2}]+3\E\left[\left|\int_{0}^{t}{\sigma(s,X_{0})dW_{s}}\right|^{2}\right]+3\E\left[\left|\int_{0}^{t}{b(s,X_{0})ds}\right|^{2}\right]$$
$$\hspace{2cm}\leq3\E[\zeta^{2}]+6\E\left[\left|\int_{0}^{t}{\sigma(s,X_{0})dB_{s}}\right|^{2}\right]+6\E\left[\left|\int_{0}^{t}{\sigma(s,X_{0})dv_{s}}\right|^{2}\right]+3\E\left[\left|\int_{0}^{t}{b(s,X_{0})ds}\right|^{2}\right]$$
$$\hspace{2cm}\leq3\E[\zeta^{2}]+6\E\left[\int_{0}^{t}{|\sigma(s,X_{0})|^{2}ds}\right]+6\E\left[\left|\int_{0}^{t}{\sigma(s,X_{0})dv_{s}}\right|^{2}\right]+3t\E\left[\int_{0}^{t}{|b(s,X_{0})|^{2}ds}\right].$$
Then,
$$\E\left[|Y^{(1)}_{t}-Y^{(0)}_{t}|^{2}\right]\leq3\E[\zeta^{2}]+6C^{2}t(1+\E[|X_{0}|^{2}])+6C^{2}\E\left[\left|(1+|X_{0}|)v_{T}\right|^{2}\right]+3C^{2}t^{2}(1+\E[|X_{0}|^{2}]).$$
Hence,
$$\E\left[|Y^{(1)}_{t}-Y^{(0)}_{t}|^{2}\right]\leq A_{0}+A_{1}t,$$
where $A_{0}=3\E[\zeta^{2}]+6C^{2}\E\left[\left|(1+|X_{0}|)v_{T}\right|^{2}\right]$ and $A_{1}=6C^{2}(1+\E[|X_{0}|^{2}])+3C^{2}T(1+\E[|X_{0}|^{2}])$. So by induction on $p$ we obtain :
$$\E\left[|Y^{(p+1)}_{t}-Y^{(p)}_{t}|^{2}\right]\leq\frac{B_{0}^{p+1}\times t^{p}}{p!}+\frac{B_{1}^{p+1}\times t^{p+1}}{(p+1)!}\text{; }p\geq0\text{, }t\in[0,T],$$
where $B_{0}$ and $B_{1}$ are some suitable constants depending only on $C,K,T,\Gamma$ and $\E\left[|X_{0}|^{2}\right]$. Now, let $\lambda$ be Lebesgue measure on $[0,T]$ and $m>n\geq0$. Hence, we get:
$$||Y^{(m)}_{t}-Y^{(n)}_{t}||_{L^{2}(\lambda\times \P)}=\left|\left|\sum_{p=n}^{m-1}{Y^{(p+1)}_{t}-Y^{(p)}_{t}}\right|\right|_{L^{2}(\lambda\times \P)}$$
$$\hspace{3.5cm}\leq\sum_{p=n}^{m-1}{\left|\left|Y^{(p+1)}_{t}-Y^{(p)}_{t}\right|\right|_{L^{2}(\lambda\times \P)}}$$
$$\hspace{3.5cm}\leq\sum_{p=n}^{m-1}{\sqrt{\E\left[\int_{0}^{T}{|Y^{(p+1)}_{t}-Y^{(p)}_{t}|^{2}dt}\right]}}$$
$$\hspace{3.5cm}\leq\sum_{p=n}^{m-1}{\sqrt{\int_{0}^{T}{\left(\frac{B_{0}^{p+1}\times t^{p}}{p!}+\frac{B_{1}^{p+1}\times t^{p+1}}{(p+1)!}\right)dt}}}.$$
Hence, we obtain:
$$||Y^{(m)}_{t}-Y^{(n)}_{t}||_{L^{2}(\lambda\times \P)}\leq \sum_{p=n}^{m-1}{\sqrt{\frac{B_{0}^{p+1}\times T^{p+1}}{(p+1)!}+\frac{B_{1}^{p+1}\times T^{p+2}}{(p+2)!}}}\longrightarrow0$$
as $n,m\longrightarrow\infty$. Therefore, $\{Y^{(n)}_{t}:n\geq0\}$ is Cauchy sequence in $L^{2}(\lambda\times \P)$. Hence, $\{Y^{(n)}_{t}:n\geq0\}$ is convergent in $L^{2}(\lambda\times \P)$. Let
$$X_{t}=\lim_{n\to\infty}{Y^{(n)}_{t}}\text{            in }L^{2}(\lambda\times \P).$$
Now, we shall show that $X_{t}$ is solution of \eqref{eds1}. We have $\forall n\geq0$, and all $t\in[0,T]$,
$$Y^{(n+1)}_{t}=\zeta1_{H}(t)+\left[Z+\int_{0}^{t}{\sigma(s,Y^{(n)}_{s})dB_{s}}+\int_{0}^{t}{\sigma(s,Y^{(n)}_{s})dv_{s}}+\int_{0}^{t}{b(s,Y^{(n)}_{s})ds}\right]1_{H^{c}}(t).$$
But as $n\to\infty$, we obtain from the Hölder inequality that
$$\int_{0}^{t}{b(s,Y^{(n)}_{s})ds}\longrightarrow\int_{0}^{t}{b(s,X_{s})ds}\text{ }\text{ in }L^{2}(\lambda\times \P).$$
And through Itô's isometry, we get:
$$\int_{0}^{t}{\sigma(s,Y^{(n)}_{s})dB_{s}}\longrightarrow\int_{0}^{t}{\sigma(s,X_{s})dB_{s}}\text{ }\text{ in }L^{2}(\lambda\times \P).$$
Furthermore, we have:
$$\left|\left|\int_{0}^{t}{\sigma(s,X_{s})dv_{s}}-\int_{0}^{t}{\sigma(s,Y^{(n)}_{s})dv_{s}} \right|\right|^{2}_{L^{2}(\lambda\times \P)}=\E\left[\left|\int_{0}^{t}{[\sigma(s,X_{s})-\sigma(s,Y^{(n)}_{s})]dv_{s}} \right|^{2}\right]$$
$$\hspace{3cm}\leq K^{2}\E\left[\left|\int_{0}^{t}{|X_{s}-Y^{(n)}_{s}|dv_{s}} \right|^{2}\right].$$
But, $\int_{0}^{t}{|X_{s}-Y^{(n)}_{s}|dv_{s}}=0$ since $\forall s\in H$, $X_{s}=Y^{(n)}_{s}=\zeta$ and $dv$ is carried by $H$. Then,
$$\int_{0}^{t}{\sigma(s,Y^{(n)}_{s})dv_{s}}\longrightarrow\int_{0}^{t}{\sigma(s,X_{s})dv_{s}}\text{ }\text{ in }L^{2}(\lambda\times \P).$$
Consequently, $\forall t\in[0,T]$ we have
$$X_{t}=\zeta1_{H}(t)+\left[Z+\int_{0}^{t}{\sigma(s,X_{s})dB_{s}}+\int_{0}^{t}{\sigma(s,X_{s})dv_{s}}+\int_{0}^{t}{b(s,X_{s})ds}\right]1_{H^{c}}(t).$$
\end{proof}

\begin{lem}\label{l7}
Let $X$ and $Y$ be solutions of \eqref{eds} such that $\forall t\in H$, $X_{t}=Y_{t}$. If $X_{0}=Y_{0}$ hence, $X$ and $Y$ are indistinguishable.
\end{lem}
\begin{proof}
We have $\forall t\geq0$,
$$\E\left[|Y_{t}-X_{t}|^{2}\right]\leq3\E\left[\left|Y_{0}-X_{0}\right|^{2}\right]+3\E\left[\left|\int_{0}^{t}{(\sigma(s,Y_{s})-\sigma(s,X_{s}))dW_{s}}\right|^{2}\right]+3\E\left[\left|\int_{0}^{t}{(b(s,Y_{s})-b(s,X_{s}))ds}\right|^{2}\right].$$
But, 
$$\E\left[\left|\int_{0}^{t}{(\sigma(s,Y_{s})-\sigma(s,X_{s}))dW_{s}}\right|^{2}\right]\leq2\E\left[\left|\int_{0}^{t}{(\sigma(s,Y_{s})-\sigma(s,X_{s}))dB_{s}}\right|^{2}\right]+2\E\left[\left|\int_{0}^{t}{(\sigma(s,Y_{s})-\sigma(s,X_{s}))dv_{s}}\right|^{2}\right]$$
We obtain from Itô isometry and Cauchy-Swarz's inequality the following:
$$\E\left[\left|\int_{0}^{t}{(\sigma(s,Y_{s})-\sigma(s,X_{s}))dW_{s}}\right|^{2}\right]\leq2\int_{0}^{t}{\E\left[|\sigma(s,Y_{s})-\sigma(s,X_{s})|^{2}\right]ds}+2\E\left[v_{t}\int_{0}^{t}{|\sigma(s,Y_{s})-\sigma(s,X_{s})|^{2}dv_{s}}\right].$$
Hence, according to Lipschitz property, we get:
$$\E\left[\left|\int_{0}^{t}{(\sigma(s,Y_{s})-\sigma(s,X_{s}))dW_{s}}\right|^{2}\right]\leq2K^{2}\int_{0}^{t}{\E|Y_{s}-X_{s}|^{2}ds}+2K^{2}\E\left[v_{t}\int_{0}^{t}{|Y_{s}-X_{s}|^{2}dv_{s}}\right].$$
Which implies that
$$\E\left[\left|\int_{0}^{t}{(\sigma(s,Y_{s})-\sigma(s,X_{s}))dW_{s}}\right|^{2}\right]\leq2K^{2}\int_{0}^{t}{\E|Y_{s}-X_{s}|^{2}ds}.$$
Indeed $\int_{0}^{t}{|Y_{s}-X_{s}|^{2}dv_{s}}=0$ since $dv$ is carried by $H$ and $Y-X=0$ on $H$. On another hand, through Cauchy-Shwarz's inequality and Lipschitz property, we get:
$$\E\left[\left|\int_{0}^{t}{(b(s,Y_{s})-b(s,X_{s}))ds}\right|^{2}\right]\leq tK^{2}\int_{0}^{t}{\E|Y_{s}-X_{s}|^{2}ds}.$$
Then, we obtain the following:
$$\E\left[|Y_{t}-X_{t}|^{2}\right]\leq3\E\left[\left|Y_{0}-X_{0}\right|^{2}\right]+3K^{2}(2+t)\int_{0}^{t}{\E|Y_{s}-X_{s}|^{2}ds}.$$
Thus, the function $\varphi$ defined by $\forall t\in[0,T]$, $\varphi(t)=\E\left[|Y_{t}-X_{t}|^{2}\right]$ satisfies,
$$\varphi(t)\leq F+A\int_{0}^{t}{\varphi(s)ds},$$
where $F=3\E\left[\left|Y_{0}-X_{0}\right|^{2}\right]$ and $A=3K^{2}(2+t)$.
By, the Gronwall Inequality, we get that 
$$\varphi(t)\leq F\exp(At).$$
Now, assume that $X_{0}=Y_{0}$. That is, $F=0$. And then, $\varphi(t)=0$ for all $t\in[0,T]$. Consequently, $X=Y$ a.s.
\end{proof}

\begin{coro}
Under assumptions $H_{1}$, the stochastic differential equation \eqref{eds1} has a unique solution.
\end{coro}
\begin{proof}
Let $Y$ and $X$ be two solutions of \eqref{eds1}. Hence, $\forall t\geq0$, $Y_{t}=X_{t}=\zeta$. Hence, $X$ and $Y$ are solutions of \eqref{eds} such that $X=Y$ on $H$. Consequently, we obtain the result by applying Lemma \ref{l7}
\end{proof}


{\color{myaqua}}

\begin{thebibliography}{10}
{\color{black}

\bibitem{Akdim}
K.~Akdim, M.~Eddahbi, M.~Haddadi, 
\newblock Characterization of sub-martingales of a new class \texorpdfstring{$(\Sigma^{r})$}{}. \newblock {\em Stoch. Anal. Appl.}, DOI:10.1080/07362994.2018.1429932.

\bibitem{1}
J.~Azéma, M.~Yor,
\newblock Sur les zéros des martingales continues, \newblock {\em Séminaire de probabilités (Strasbourg)} 26 (1992) 248--306.

\bibitem{sak}
S.~Beghdadi-Sakrani.
\newblock Calcul stochastique pour les mesures signées. \newblock {\em Séminaire de probabilités (Strasbourg)}, 36: 366-382, 2003.
\bibitem{siam}
S.~Bouhadou and Y.~Ouknine.
\newblock On the time inhomogeneous skew Brownian motion. \newblock {\em Bulletin des Sciences Mathématiques}, vol.137(7): 835-850, 2013.
\bibitem{pat}
P.~Cheridito, A.~Nikeghbali, E.~Platen,
\newblock Processes of class sigma, last passage times, and drawdowns, \newblock {\em SIAM J. Financ. Math.}, 3(1) (2012) 280--303.

\bibitem{skews}
F.~Eyi Obiang,
\newblock Resolution of the skew Brownian motion equations with stochastic calculus for signed measures, \newblock {\em  Stochastic Analysis and Applications}, https://doi.org/10.1080/07362994.2020.1844022.

\bibitem{f}
F.~Eyi-Obiang, Y.~Ouknine and O.~Moutsinga.
\newblock New classes of processes in Stochastic calculus for signed measures. \newblock {\em Stochastics}, 86(1): 70-86, 2014.

\bibitem{fjo}
F.~Eyi Obiang, O.~Moutsinga and Y.~Ouknine.
\newblock An ideal class to construct solutions for skew Brownian motion equations. \newblock {\em Journal of Theoretical Probability},DOI:10.1007/s10959-021-01078-5,2021

\bibitem{eomt}
F.~Eyi Obiang, Y.~Ouknine, O.~Moutsinga, G.~Trutnau,
\newblock Some contributions to the study of stochastic processes of the classes  \texorpdfstring{$\Sigma(H)$}{} and \texorpdfstring{$(\Sigma)$}{}, \newblock {\em  Stochastics}, 89, 8:1253-1269, 2017.

\bibitem{pot}
C.~Dellacherie, P.A.~Meyer,
\newblock Probabilités et Potentiel. \newblock {\em Chapitres V à VIII. Théorie des Martingales. Revised Edition, Hermann, 1980, Paris}
\bibitem{jeu}
T.~Jeulin.
\newblock semimartingales et grossissement d'une filtration. \newblock {\em Lecture notes in mathematics. Springer},  1980.
721: 478--487, 1979.
\bibitem{ruin}
S.~Hussain, A.~Zeb, K.~Javed, I.~Khan.
\newblock The Ruin Probability Function of an Investment Strategy for Insurance Companies. \newblock {\em To appeared in Aim Press}.
721: 478--487, 1979.
\bibitem{mey}
P.A.~Meyer, C.~Stricker, M.~Yor
\newblock Sur une formule de la théorie du balayage, \newblock {\em in: Sém.proba. XIII, in: Lecture Notes in
Mathematics}, 721: 478--487, 1979.
\bibitem{nik}
A.~Nikeghbali, 
\newblock A class of remarkable sub-martingales, \newblock {\em J. Theor. Probab.} 4(19) (2006) 931--949.
\bibitem{mult}
A.~Nikeghbali,
\newblock Multiplicative decompositions and frequency of vanishing of nonnegative sub-martingales, \newblock {\em J. Theor. Probab.} 19(4) (2006) 931--949.
\bibitem{bo}
B.~Oksendal, 
\newblock Sochastic Differential Equations, An Introduction with Applications, \newblock {\em fifth ed., Springer-Verlag, New York}, 1998.
\bibitem{chav}
J.~Ruiz de Chavez.
\newblock Le théorème de Paul Lévy pour des mesures signées. \newblock {\em Séminaire de probabilités (Strasbourg)}, 18: 245-255, 1984.
\bibitem{sak1}
S.~Sakrani.
\newblock Representation of martingales under signed measures and the study of the classes \texorpdfstring{$\Sigma_{s}$}{} and \texorpdfstring{$\Sigma_{s}^{'}$}{}. \newblock {\em Stochastics, }, 93: 196-210, 2021.
\bibitem{y}
M.~Yor.
\newblock Sur le balayage des semimartingales continues. \newblock {\em Séminaire de probabilités (Strasbourg)}, 13: 453-471, 1979.
\bibitem{y1}
M.~Yor,
\newblock Les inégalités de sous-martingales, comme conséquences de la relation de domination, \newblock {\em Stochastics} 3(1) (1979) 1--15.

}
\end{thebibliography}
\end{document}